%% file: IETRevArX.tex
\documentclass[11pt]{amsart}
\usepackage{amscd, amsthm}
\usepackage{amsmath,amsfonts,amsthm,epsfig,latexsym,graphicx,amssymb,psfrag}
\usepackage{amssymb,amsthm,amsmath,xspace}
\usepackage[english]{babel}
\usepackage{xcolor}
\usepackage{enumerate}
\usepackage{array} 
\usepackage{float}
\newtheorem{coro}{Corollary}
\newtheorem{prop}{Proposition}[section]
\newtheorem{theo}{Theorem}
\newtheorem{lemm}{Lemma}[section]

\newtheorem*{cons}{Consequences of Lemma 3.2}

\newtheorem{clai}{Claim}

\newtheorem{remar}{Remark}[section]

\newtheorem{propri}{Properties}[section]

 \newtheoremstyle{rem}
 {3 pt}{3 pt}{\rm}{3 pt}{\bf}{\hskip -0.5 pt . }{0 pt}
 {\thmname{#1}\thmnumber{#2}\thmnote{\textnormal(#3)}}
 \theoremstyle{rem}
 \newtheoremstyle{defi}
 {3 pt}{3 pt}{\rm}{3 pt}{\bf}{\hskip -0.5 pt . }{0 pt}
 {\thmname{#1}\thmnumber{#2}\thmnote{\textnormal(#3)}}
 \theoremstyle{defi}
 \newtheorem{defi}{Definition }[section]
 
\title{Reversible maps and products of involutions in groups of IETs.}

\author{ Nancy Guelman and Isabelle Liousse}

\begin{document}

\address{{\bf  Nancy  GUELMAN}, IMERL, Facultad de Ingenier\'{\i}a, {Universidad de la Rep\'ublica, C.C. 30, Montevideo, Uruguay.} \emph {nguelman@fing.edu.uy}.}

\address{{\bf Isabelle LIOUSSE}, Univ. Lille, CNRS, UMR 8524 - Laboratoire Paul Painlevé, F-59000 Lille, France. \emph {isabelle.liousse@univ-lille.fr}.}
\begin{abstract}
 
An element $f$ of a group $G$ is reversible  if it is conjugated in $G$ to its own inverse; when the conjugating map is an involution, $f$ is called strongly reversible.
We describe reversible maps  in certain groups of interval exchange transformations namely $G_n \simeq  (\mathbb S^1)^n  \rtimes\mathcal S_n $, where $\mathbb S^1$ is the circle  and $\mathcal S_n $ is the group of permutations of $\{1,...,n\}$. We first characterize strongly reversible maps, then we show that reversible elements are strongly reversible. As a corollary,  we obtain that composites of involutions  in $G_n$ are product of at most four involutions.

We prove that any reversible Interval Exchange Transformation (IET) is  reversible by a finite order element and then it is the product of two periodic IETs. In the course of proving this statement, we classify the free actions of $BS(1,-1)$ by IET and we extend this classification to free actions of finitely generated torsion free groups containing a copy of $\mathbb Z^2$. We also give examples of faithful free actions of $BS(1,-1)$ and other groups containing reversible IETs.

We show that periodic IETs are product of at most $2$ involutions.  For IETs that are products of involutions, we show that such 3-IETs are periodic and then are product of at most $2$ involutions and  we exhibit a family of non periodic 4-IETs for which we prove that this number is at least $3$ and at most $6$. \vskip -2truecm
\end{abstract}

\maketitle
\vskip -2truecm
\tableofcontents

\section{Introduction.}

An element  $f$ of a group $G$ is called {\bf reversible} if there exists $b \in G$ such that $bfb^{-1}=f^{-1}$. We say that the element $b$  {\bf reverses}  $f$ and $f$ is said to be {\bf $b$-reversible.}

An element  $f$ is called {\bf strongly reversible} if there exists an involution  $b \in G$ such that $bfb=f^{-1}$.

\smallskip

An {\bf interval exchange transformation (IET)}  is a bijective map $f: [0,1) \rightarrow [0,1)$ defined by a finite partition of the unit interval into half-open subintervals and a  reordering of these intervals by translations. We denote by $\mathcal G$  {\bf the group consisting in all IETs}. If the partition has cardinal $r$, we say that $f$ is an $\mathbf{r}${\bf -IET}. 

More generally, an {\bf affine interval exchange transformation (AIET)} is a bijective map $[0,1) \rightarrow [0,1)$ defined by a finite partition of the unit interval into half-open subintervals such that the restriction to each  of these intervals is a  direct  affine map.

\smallskip

In \cite{GL}, the authors  have proved that the Baumslag-Solitar groups (see \cite{BS}) defined by $BS(m,n)= \langle\ a,b \ \vert \  ba^m b^{-1}=a^n\ \rangle$ do not act faithfully by  AIET, when $|m| \neq |n|$. So, it is natural to ask whether $BS(1,-1)$ acts faithfully by AIET or  IET. In section 2, we will construct examples of free faithful actions of $BS(1,-1)$ by IET.

\smallskip

Therefore, it is of interest to classify reversible maps in AIET or IET. O'Farrell and Short, motivated by the importance of reversibility in dynamical systems and group theory, wrote a quite complete survey (see \cite{FS}) exposing mostly recent works on the problem of classifying reversible elements in a large number of groups. The authors have  raised the following questions: given a group $G$, are all reversible elements of $G$  strongly reversible? Does there exist bounds for the minimal number of involutions that are needed for writing an element in the subgroup of $G$ generated by its involutions.

In this paper, we first answer the O'Farrel-Short questions for some particular groups of interval exchange transformations  $G_n \simeq  (\mathbb S^1)^n  \rtimes\mathcal S_n$, where $\mathcal S_n  $ is the group of permutations of $\{1,...,n\}$.  Later, we answer {partially} these questions in $\mathcal G$,  noting that the groups $G_n$ play an important role. These groups $G_n$ will be precisely  defined in the following

\begin{defi}  \label{Gn} \  

Let $n$ be a positive integer and $\mathbb S_n= [0, \frac{1}{n}]/0=\frac{1}{n}$ the circle of length $\frac{1}{n}$. We define $G_n$ as the set of IETs on $[0,1)$ that preserve the partition $[0,1)=[0,\frac{1}{n})\cup [\frac{1}{n}, \frac{2}{n}) ... \cup [\frac{n-1}{n},1)$ and whose restrictions to intervals $I_i=[\frac{i-1}{n}, \frac{i}{n})$ are IETs with only one interior discontinuity.

\smallskip

For $g\in G_n$, we define $\sigma_g$ as the element of $\mathcal S_n$ given by $\sigma_g(i)=j$, if $g(I_i) = I_j$. 

It follows that   $g_{\vert I_i} = R_{\alpha_i,\sigma_g(i)}$, where $R_{\alpha_i,\sigma_g(i)}(x) =  x + \alpha_i +\frac{\sigma_g(i)-i}{n} \  (mod \frac{1}{n})$.

\smallskip

We define $\alpha_g=(\alpha_1, ..., \alpha_n)=( \alpha_1(g), ..., \alpha_n(g)) \in {\mathbb S_n} ^n $ and we denote $g=(\alpha_g, \sigma_g)$.
\end{defi}

\begin{remar}\label{cons} {A straightforward consequence of this definition is that the map $( \alpha,\sigma ): G_n \rightarrow ({\mathbb S}_n)^n \rtimes\mathcal S_n  $ is an isomorphism,  thus the group $G_n$  is virtually abelian. It can be seen as a subgroup of $GL(n,\mathbb C)$ and any finitely generated virtually abelian group is a subgroup of some $G_n$. These groups are very relevant since Dahmani, Fujiwara and Guirardel (see  \cite{DFG}) recently proved that any finitely generated torsion free solvable subgroup of IETs is virtually abelian.}
\end{remar}

On one hand, according to Proposition 3.4 of \cite{FS} any permutation is strongly reversible in $\mathcal S_n$. On the other hand, reversible elements in abelian groups are involutions. Groups $G_n$ are somewhat intermediate.
 
\medskip

Given $\sigma$ and $\tau$ that reverses $\sigma$, in this paper we firstly determine  necessary and sufficient conditions on  $\alpha_f$ such that $f=(\alpha_f, \sigma)$ is reversible in $G_n$  by an involution of the form $ T=(\alpha_T, \tau)$ and we describe the possible $\alpha_T$'s.

\smallskip

\begin{defi}
Consider $\sigma$ and $\tau$ that reverses $\sigma$, let us denote $S_i$ the {$\sigma$}-cycle by $i$. An element $u$ is called {\bf distinguished} if either $ \tau(u) \notin S_u$ or $ \tau(u) \in \{u,\sigma(u)\}$. 
\end{defi}
In Lemma \ref{cyc}, we will prove any $\langle\sigma_f, \tau \rangle$-orbit has distinguished elements.
We can now formulate our first result that characterizes strongly reversible elements in $G_n$.

\begin{theo}\label{Th1}
Let $f=(\alpha_f,\sigma_f)\in G_n$, $\tau$ an involution that reverses $\sigma_f$ and 
 $D$  a set of distinguished representatives of the $\langle\sigma_f, \tau \rangle$-orbits.

\smallskip

There exists an involution $T=(\alpha_T,\tau)$ that reverses $f$ if and only if for any $u$ in $D$, 

\begin{equation} 
\tag{$\dagger \dagger$}
 \sum_{j\in S_u} \alpha_{j}(f)+\displaystyle\sum_ {j\in \tau(S_u)} \alpha_{j}(f)=0.
 \label{CNSu}
 \end{equation} 

Under these conditions,
\begin{enumerate}[(A)]
\item if $S_u$ is  $\tau$-invariant then  $\alpha_{u}(T)$ is given by  
\begin{equation*}
 2 \alpha_{u}(T)=0 {\text  { \ if \  }} \tau(u)=u
{\text { \ \ \ \  or \  \ \ \ }} 
 2 \alpha_{u}(T)= 2\alpha_{u}(f)   {\text { \ if \ } } \tau(u)=\sigma_f(u).
\end{equation*} 
\noindent and it determines uniquely $\alpha_{j}(T)$ for all $j\in S_u $;

\medskip
 
\item  if  $S_u$ is not $\tau$-invariant then $\alpha_u(T)$ can be chosen arbitrary and this choice determines uniquely $\alpha_{j}(T)$ for all $j\in S_u \cup \tau(S_u)$.
\end{enumerate}
\end{theo}

\smallskip

\begin{remar} {More precisely}, under an admissible choice of $\alpha_u(T)$ for $u\in D$,  the other coordinates of $\alpha(T)$ are uniquely defined by induction using the identities $\displaystyle  (1_k) \ \ \alpha_{\tau(u^k)} (T) = -\alpha_{u^k} (T) $  and $\displaystyle (4_k) \ \  \alpha_{u^k} (T) = \alpha_u(T)- \sum_{j=0}^{k-1} \alpha_{u^j} (f) -  \sum_{j=1}^{k} \alpha_{\tau (u^j)} (f) $, \     where $u^k =\sigma_f^k(u)$. As consequences of Lemma \ref{alpha}, we will see that these identities are necessary and they will be used in the proof of Theorem \ref{Th1} {(2)} in order to establish that Conditions (\ref{CNSu}) are sufficient to construct a reversing involution of $f$ with associated permutation $\tau$.

In addition if $S_u$ has an odd length then distinguished representatives of both nature exist and give rise to apparently distinct conditions but they are equivalent by  $(4_k)$.
\end{remar}

\smallskip

The last theorem allows to prove the following
\begin{coro} \label{FO}\ 

\begin{enumerate}
\item  A strongly reversible element $f$ of $G_n$ having a cycle as associated permutation $\sigma_f$  (or a product of cycles with distinct lengths) has finite order and there exist finite order elements of $G_n$ that are not strongly reversible in $G_n$.
\item There exist strongly reversible elements that are not of finite order, however  minimal strongly reversible elements do not exist in $G_n$.
\end{enumerate}
\end{coro}

{Summing Conditions (\ref{CNSu}) over $u\in D$, we get that} any strongly reversible $f\in G_n$ satisfies {$2\sum_{j=1}^n$}  $ \alpha_{j}(f)=0 $. So, it is convenient to define ${\mathbf{A(f)}}= 2\sum_{j=1}^n \alpha_{j}(f)$ for $f \in G_n$. It is easy to check that $A$ is a morphism and $A(f)=0$ for $f$ strongly reversible in $G_n$.

\begin{remar} \label{revcyc}
When $\sigma_f$ is a $n$-cycle, $f$ is strongly reversible if and only if $A(f)=0$ and there are  two possible admissible choices for $\alpha_u(T)$ (since it is well defined modulo $1/2$). 
\end{remar}  

\begin{coro}\label{PI} Let $I_n$ be the normal subgroup of $G_n$ generated by its involutions.
\begin{enumerate}
\item  The kernel of $A$ coincides with $I_n$.
\item  Any $f \in I_n$ can be written as the product of at most $4$ involutions.
\item  For any $n\geq 3$, there exists $f\in G_n$ which can not be written as a product of $3$ involutions.
\end{enumerate}
 \end{coro}
  
Theorem \ref{Th1} provides a characterization of strongly reversible maps in $G_n$. It turns out that it also holds for reversible elements according to next

\begin{theo}\label{Th2}
Let $f=(\alpha_f,\sigma_f)\in G_n$ reversible in $G_n$ then $f$ is strongly reversible in $G_n$.
\end{theo}

\medskip

 For $\mathcal G$ we show a similar result. For that, we describe the free faithful actions of  $BS(1,-1)= \langle \ a,b \ \vert \  ba b^{-1}=a^{-1} \ \rangle$ by IET.  We recall that a group is said to {\bf act freely} if the only element having some fixed point is the trivial element.  In section 2, an example of such an action of  $BS(1,-1)$ by  elements of $G_4$ is given.

\medskip

\begin{theo} \label{Th3}
Let $\langle f,h\rangle$ be a free faithful action of $BS(1,-1)$ by IET then $\langle f,h\rangle$ is $PL\circ IET$-conjugated (that is conjugated through $g= R \circ E $, where $R$ is a PL-homeomorphism and $E$ is an $IET$) to a  free  action of $BS(1,-1)$ by elements of some $G_n$. \end{theo}

In Lemma \ref{FNBS}, it is shown that a  $BS(1,-1)$  action  is faithful if and only if $f$ and $h$ have infinite order therefore it is natural to deal with the case where $f$ is a minimal IET. As a consequence of next statement, there does not exist free faithful actions of $BS(1,-1)$ by IET with $f$ minimal.

\begin{theo} \label{Th4.1} \ 
\begin{enumerate}
\item  If an IET $f$ is minimal and reversible by $h$ then $\langle f,h\rangle$ is free.
\item  There does not exist a minimal IET  that is reversible by an infinite order $h$. 
\item  Any periodic IET is strongly reversible.
\end{enumerate}
\end{theo}
Note that Item (3)  is not true in $G_n$ (see item (1) of Corollary \ref{FO}).

The last two results will be combined with  a dynamical decomposition (see Proposition \ref{deco})  for proving

\begin{theo}\label{Th4}
Let $f\in \mathcal G$ reversible in $\mathcal G$ then $f$ is  reversible in $\mathcal G$ by a finite order element $h$. Moreover $h$ is either an involution or its order is a multiple of $4$.
\end{theo}

\begin{coro} \label{corpiper}
Any reversible IET can be written as a product of $2$ periodic elements.
\end{coro}

Related to the second question of O'Farrell-Short, it has been proved by Vorobets (\cite{Vo}) that the subgroup of $\mathcal G$ generated by its involutions coincides with $\mathcal G_0$, the subgroup of  $\mathcal G$ consisting in IETs having zero SAF-invariant (definition and properties of  SAF-invariant will  be given in section 8). Unfortunately, Vorobets tools do not give upper bounds for the number of involutions that are needed.  Theorem \ref{Th4.1} Item (3) and Corollary \ref{corpiper} directly imply the following 
\begin{coro} \label{perrev}
\item Any periodic element of $\mathcal G$ is the product of at most $2$ involutions. 
\item Any reversible IET can be written of $4$ involutions.
\end{coro}

Finding an upper bound for the number of involutions that are needed to write non periodic r-IETs of $\mathcal G_0$ with $r>3$ is a delicate problem. As indicated in item (1) of the following theorem, our   first example of a non strongly reversible element of $\mathcal G_0$ is a product of two restricted rotations of pairwise disjoint supports.

\begin{defi}
An  IET $g$ whose support is an interval $I$ is a {\bf restricted rotation} if there exists a direct affine map from $I$ to $[0,1)$  that conjugates $g_{\vert I}$ to a 2-IET.
\end{defi}

In section 7, we show that if a product of two restricted rotations of supports $I_1$ and $I_2=[0,1)\setminus I_1$ belongs to  $\mathcal G_0$ then $\vert I_1 \vert$ is a rational number, where  $\vert J \vert$ denotes the length of the interval $J$.

Under an additional assumption on $\vert I_1\vert$, item (2) of next Theorem gives a bound for the number of involutions.

\begin{theo}\label{Th7} \ 
\begin{enumerate}
\item Any $3$-IET  that belongs to $\mathcal G_0$ is periodic so is the product of two involutions.
\item There exists  a 4-IET  $f\in \mathcal G_0$ that is not reversible. 
\item Any element of $\mathcal G_0$ that is a product of 2 restricted rotations with respective supports $I_1$ and $I_2=[0,1)\setminus I_1$, satisfying $\vert I_1 \vert = \frac{p}{p+1}$ with $p\in \mathbb N^*$, can  be written as a product of $6$ involutions.
\end{enumerate}
\end{theo} 

\medskip

In the last section of this paper, we are interested in actions by IET of some torsion-free groups containing reversible elements. Using properties of the Poincar\'e rotation number, it is easy to check that a group containing an infinite order reversible element can not act freely by circle homeomorphisms. This is no longer true for actions by IET. More precisely, we prove  

\begin{theo} \label{related} \
\begin{enumerate}
\item The Baumslag-Solitar group $BS(1,-1)$ acts freely and minimally by IET.
\item The crystallographic group $C_1 =\langle a, b \ | \ b a^2  b^{-1} = a^{-2},  a b^2 a^{-1} =b^{-2}\rangle$ acts freely and minimally by IET.
\item The group $C_2 =\langle a, b, c \  | \ ab a^{-1} = b^{-1}, bcb^{-1} = c^{-1} \rangle$  does not admit faithful  minimal action by IET.
\end{enumerate}
\end{theo}

In item (3), since $\langle a,c\rangle$ is isomorphic to the non abelian free group of rank two $\mathbb F_2$ (see \cite{Ler}), the existence of faithful actions of $C_2$ by IET is related to Katok's question: Does $\mathcal G$ contain a subgroup isomorphic to $\mathbb F_2$ ?  A partial answer was given by Dahmani, Fujiwara and Guirardel: the group generated by a generic pair of IETs is not free (\cite{DFG1}). 

Notice that the actions constructed in items (1) and (2) are by elements of $G_4$. 

In addition any finitely generated group acting freely by IET is a subgroup of some  $G_n$: indeed  since translations commute, the orbit of any point under a finitely generated group $H$ of IETs has polynomial growth  therefore if $H$ acts freely it has polynomial growth so it is virtually nilpotent by \cite{Gro}. According to Novak \cite{No}, $H$ is virtually abelian. Finally, by Remark \ref{cons} $H$ is isomorphic to a subgroup of some $G_n$. In particular, neither $\mathbb F_2$ or $C_2$ can act freely by IET.

\medskip

We will prove a dynamical rigidity result for free actions by IET:
\begin{theo} \label{extfree}
Let $G$ be  a  finitely generated torsion free group that contains a copy of $\mathbb Z^2$.

If  $\rho : G \rightarrow \mathcal G$  is a free faithful action of $G$ by IET then  the image $\rho (G)$ is $PL\circ IET$-conjugated to a subgroup of some $G_n$.
\end{theo}

\noindent \textbf{Acknowledgements.} We thank Juan Alonso for helpful discussions concerning Theorem \ref{Th1}. We acknowledge support from the MathAmSud Project GDG 18-MATH-08, the Labex CEMPI (ANR-11-LABX-0007-01), the University of Lille (BQR) and the I.F.U.M.


\section{Examples and general properties.}
We let the reader check the following basics facts of reversibility.
\subsection{Examples. } 
\subsubsection{Basic examples.}
Involutions are strongly reversible and  any product of two involutions $f=\sigma_1 \sigma_2$ is strongly reversible by $\sigma_1$ and $\sigma_2$.

\subsubsection{Actions of $BS(1,-1)$ by IET} The following elements  $a$ and $b$ of $G_4$ generate a faithful and free action of $BS(1,-1)$ provided that $\alpha$ and $\beta$ are rationally independent  irrational numbers. The graph of an element $f$ of $G_n$ is represented in $[0,1]\times [0,1]$ by indicating  in each square $ I_i \times I_{\sigma(i)}$ the corresponding angle $\alpha_i(f)$.
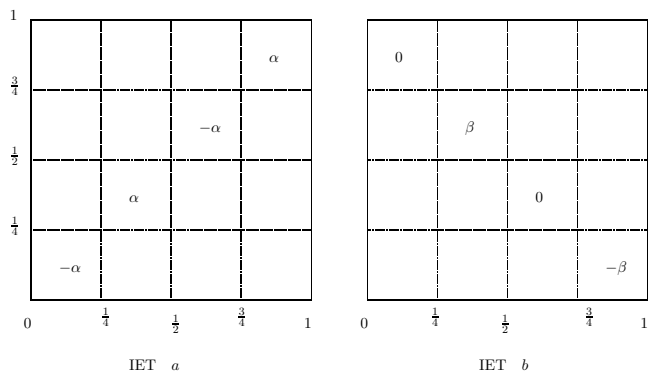
\begin{figure}[H]
    \scalebox{.53}{\input{BS.tex}}
\end{figure} \ \  

\subsection{Properties.} \label{Pro2}
\begin{itemize}
\item An element $f$ is strongly reversible if and only if $f=\sigma_1 \sigma_2$ with $\sigma_1^2=\sigma_2^2=Id$.
\item If $ h_1$ and $h_2$ reverse the same $f$ then $h_2^{-1} h_1 $ commutes with $f$.
\item If  $f$ is $h$-reversible  then $f$ is $(hf^s)$-reversible for all $s\in \mathbb Z$.
\item If  $f$ is $h$-reversible  [resp. strongly $h$-reversible] then  $f^n$ is $h$-reversible  [resp. strongly $h$-reversible], for all $n\in \mathbb Z$.
\item If  $f$ is $h$-reversible then $h^{2p}$ commutes with $f$ and $h^{2p+1}$ reverses $f$,  for all $p\in \mathbb Z$.
\end{itemize}


\section{Strong Reversibility in $G_n$. Proof of Theorem  \ref{Th1}.}
\subsection{Necessary preliminaries.} 
\subsubsection{Action of the symmetric group $\mathcal S_n$ on vectors.}

The symmetric group $\mathcal S_n$ acts on $\mathbb R^n$ by permuting the coordinates leading to an action on the quotient space $\mathbb{S}_n ^n$, where $\mathbb{S}_n=[0,\frac{1}{n}]/{}_{0=\frac{1}{n}}$. 

\begin{defi} 
Given $\sigma \in \mathcal S_n$ and $\alpha=(\alpha_1,  \cdots , \alpha_n) \in \mathbb{S}_n^n$, we define $$\sigma(\alpha) = (\alpha_{\sigma(1)}, \cdots , \alpha_{\sigma(n)}).$$ We denote by $M_ \sigma$ its associated matrix with respect to the canonical basis of $\mathbb R^n$.
\end{defi}

\smallskip

\begin{clai} Let $\alpha=(\alpha_1, \cdots , \alpha_n)\in \mathbb{S}_n^n$ and $\sigma, \tau \in \mathcal S_n$. Then $\displaystyle (\sigma \tau ) (\alpha) = \tau (\sigma(\alpha))$. In other words, $\displaystyle M_{\sigma \tau} =M_\tau M_\sigma$.
\end{clai}

\begin{proof}
By definition, $\sigma(\alpha) = (\alpha_{\sigma(1)},  \cdots , \alpha_{\sigma(n)})=(\beta_1,...,\beta_n)$ and 

$\tau (\sigma(\alpha)) = (\beta_{\tau(1)}, \cdots , \beta_{\tau(n)})$ with $\beta_i=\alpha_{\sigma(i)},$ so $\beta_{\tau(j)}=\alpha_{\sigma(\tau(j))}$.

Finally, $\tau (\sigma(\alpha)) =(\alpha_{\sigma(\tau(1))}, \cdots,\alpha_{\sigma(\tau(n))}) = \sigma \tau (\alpha)$. \end{proof}

\subsubsection{Strong reversibility of $\sigma$.}
Next lemma describes $\langle\sigma,\tau\rangle$-orbits and shows existence of distinguished elements in each $\langle\sigma,\tau\rangle$-orbit.

\begin{lemm} \label{cyc}
Let $f \in G_n$ be $T$-strongly reversible in $G_n$. Then
\begin{itemize}
\item $\sigma_f$ is $\sigma_T $-strongly reversible in $\mathcal S_n$.
\item If $i\in \{1,...,n\}$ belongs to a $\sigma_f$-cycle $S$ of  length $p$ then either
\begin{itemize}
\item $\sigma_T(i)$ belongs to $S$ and there exists $u$ in $S$ such that $\sigma_T(u)\in \{u, \sigma_f(u)\}$ or
\item $\sigma_T(i)$ generates a disjoint $\sigma_f$-cycle of  length $p$.
\end{itemize}
\end{itemize}
\end{lemm}

\noindent{\bf Proof.}

The first item  is a direct consequence of the fact that the map $f \mapsto \sigma_f$ is a morphism.

For the second item,  it is obvious  that if $\sigma_T(i)\in S$ that is if there exists $s\in \mathbb N$ such that $\displaystyle \sigma_T(i)=\sigma_f^{\ s} (i)$ then  $i$ and  $\sigma_T (i)$ generate the same cycle of $\sigma_f$.

In addition, for any $k$ one has $\sigma_T (\sigma_f^{\ k}(i))= \sigma_{f}^{-k} ( \sigma_T (i))=
\sigma_{f}^{-k+s}(i)= \sigma_{f}^{-2k+s}(\sigma_{f}^{\ k}(i))$.

Taking $u= \sigma_{f^{k}}(i)$ with $k=[\frac{s}{2}]$ the integer part of $\frac{s}{2}$, it holds that $\sigma_T(u)\in \{u, \sigma_f(u)\}$.

\medskip

The remaining case occurs when for any $j$ in  $S$, $\sigma_T(j)$ is not in $S$. By the first item, for $0\leq k< p$ we have $\sigma_f^{\ k}  (\sigma_{T} (i))= \sigma_T \sigma_{f}^{-k} (i)$ and these $p$ points do not belong to $S$.

Moreover, $\sigma_f^{\ k} (\sigma_{T} (i))= \sigma_T \sigma_{f}^{-k} (i) \not=\sigma_T(i)$, for $k=1,..., p-1$ since $\sigma_{f}^{-k}(i)\not=i$,

\noindent and $\sigma_f ^{\ p} (\sigma_T(i))= \sigma_T\sigma_{f}^{-p}(i)= \sigma_{T }(i)$.

In conclusion, $\sigma_T(i)$ generate a $\sigma_f$-cycle  of  length $p$ consisting in the $\sigma_T (j)$, $j\in S$.  \hfill $\square$

\subsubsection{Necessary and sufficient conditions for strong reversibility.} \ 

\smallskip

Consequences of Definition \ref{Gn} and straightforward calculus give rise to:
\begin{propri} \label{Pro3} Let $f$, $g$ and $T$ in   $G_n$.
\begin{enumerate}
\item $\sigma_{f\circ g} =\sigma_{f}\circ \sigma_{g}$ and $ \alpha_{f\circ g} = \alpha_g +\sigma_g( \alpha_f) $.
\item $\sigma_{f^{-1}} =\sigma_{f}^{-1}$ and $ \alpha_{f^{-1}} = -(\sigma_f)^{-1}( \alpha_f) $.
\item $ T =(\alpha_T,\sigma_{T})$ is an involution if and only if  $\sigma_{T}^2=Id$ and $ \alpha_T= -\sigma_T (\alpha_T)$  that is $\alpha_{\sigma_{T}(i)}(T)=-\alpha_i(T)$, for $1\leq i \leq n$.
\item  $ T $ reverses $f$  if and only if $\sigma_T \sigma_f \sigma_T^{-1}= \sigma_f ^{-1}$ and  $$\alpha_{T^{-1}}+\sigma_{T^{-1}}(\alpha_f) +\sigma_{T^{-1}}(\sigma_f (\alpha_T))=-\sigma_{f}^{-1}(\alpha_f).$$
\end{enumerate}
\end{propri}

We will use these properties to prove the following

 \begin{lemm}\label{alpha}
Let $f=(\alpha_f, \sigma_f)$ and $T=(\alpha_T, \sigma_T)$ in $G_n$. Then $T$ is an involution  that reverses $f$ if and only if $\sigma_T$ is an involution that reverses $\sigma_f$  and

 \begin{enumerate}
 \item $\alpha_T \in Ker (I+M_{\sigma_T})$,
 \item $\alpha_T \in Ker (I+M_{\sigma_{T}\sigma_{f}}) +\alpha_f$.
 \end{enumerate}
 \end{lemm}

\begin{proof} 

For simplicity, for $g\in G_n$ the matrix $M_{\sigma_g}$ is denoted by $M_g$. 

Suppose that $T$ is an involution  that reverses $f$. 

According to item (3) of Properties \ref{Pro3},  $\sigma_T$ is an involution that reverses $\sigma_f$ and  $\alpha_T+M_T(\alpha_T)=0$. 

It remains to prove item (2). As $T$ is an involution,  Properties \ref{Pro3} (4) leads to  $$\alpha_T+\sigma_{T}(\alpha_f) +\sigma_{T}(\sigma_f (\alpha_T))=-\sigma_{f}^{-1}(\alpha_f).$$

\medskip

 Using matrix notation, last formula implies that
$$(I + M_{T} M_f )(\alpha_T)= -(M_T +M_{f}^{-1})(\alpha_f).$$

\smallskip

Left multiplying both sides by $M_T$, we obtain
$$( M_{T} + M_f )(\alpha_T)= -( I + M_T M_{f}^{-1})(\alpha_f).$$

\smallskip

On one hand,  since $M_T(\alpha_T)= -\alpha_T$, we have that
$$( M_{T} + M_f )(\alpha_T)=(M_{T} + M_f) M_T M_T(\alpha_T)=(I + M_f M_T) (-\alpha_T)=(I + M_{T\circ f}) (-\alpha_T) .$$

\smallskip

On the other hand,  since $f^{-1}\circ T= T \circ f$, we get that
$$( I + M_T M_{f}^{-1})(\alpha_f)=( I + M_{f^{-1}\circ T})(\alpha_f)=
( I + M_{T\circ f})(\alpha_f).$$

\medskip

Combining these, we conclude that
$$(I + M_{T\circ f}) (\alpha_T) = ( I + M_{T\circ f})(\alpha_f).$$

\medskip

 Therefore  $\displaystyle \alpha_T -\alpha_f \in Ker (I +M_{T\circ f})$.

 \medskip
 Note that under the assumptions that  $\alpha_T \in Ker (I+M_{\sigma_T})$ and $\sigma_T$ reverses $\sigma_f$, all formulas below are equivalent since $M_T$ is invertible. \end{proof}

\begin{cons} Let $\tau \in \mathcal S_n$ be an involution, $f =(\alpha_f, \sigma_f)$  and $T=(\alpha_T, \tau)$ be two elements of $G_n$.

\begin{itemize}

\item The following properties (a), (b) and (c) are equivalent.

\begin{enumerate}[(a)]

\item $f$ is strongly reversible by $T$.

\medskip

\item \begin{enumerate}[$(1)$]
\item $\tau$ reverses $\sigma$, $\alpha_{\tau(j)}(T)= -\alpha_{j}(T)$ for $j\in\{1,\cdots,n\}$ and 
\item $\alpha_{\sigma(j)}(T)= \alpha_{j}(T)-\alpha_{j}(f)-\alpha_{\tau \sigma(j)}(f)$ for $j\in\{1,\cdots,n\}$.
\end{enumerate}

\medskip

\item \begin{enumerate}
\item[$(1)$]$\tau$ reverses $\sigma$,  $\alpha_{\tau(j)}(T)= -\alpha_{j}(T)$ for $j\in\{1,\cdots,n\}$ and
\item[$(2')$]$\alpha_{\tau\sigma(j)}(T)= -\alpha_{j}(T)+\alpha_{j}(f)+\alpha_{\tau \sigma(j)}(f)$  for $j\in\{1,\cdots,n\}$.
\end{enumerate}
\end{enumerate}

\medskip

\item If $f$ is strongly reversible by $T$ then for any $j\in\{1,\cdots,n\}$,

$\displaystyle (3)$  $2\alpha_{j}(T)=0$  if $\tau(j)=j$ or $2\alpha_{j}(T)= 2\alpha_{j}(f)$ if $\tau(j)=\sigma(j)$ and

$\displaystyle (4) \ \alpha_{\sigma^k(j)}(T)= \alpha_{j}(T) - \sum_{p=0}^{k-1} \alpha_{\sigma^p(j)} (f) -  \sum_{p=1}^{k} \alpha_{\tau (\sigma^p(j))} (f)$ for any $k\in \mathbb N$.
\end{itemize}
\end{cons}

\begin{proof} \ 

\begin{itemize}
\item  The equivalence between Items (a) and (c) is exactly the statement of Lemma 3.2.  Under Condition (1), Conditions (2)  and (2') are equivalent since $\alpha_{\tau\sigma(j)}(T)= -\alpha_{\sigma(j)}(T)$. So Items (b) and (c) are also equivalent.

\item Condition (3) is a direct consequence of  (1) and (2').

\noindent Condition (4) is proved by induction on  $k$.

Indeed, for $k=0$ we get $\alpha_{j}(T)= \alpha_{j}(T)$ (under the convention that $\sum_a ^b=0$  if $b<a$) and for $k=1$, Condition (4) is  Condition (2). We suppose that $(4)$ holds for a given arbitrary positive integer $k$, using Condition (2) one has:
$$\alpha_{\sigma^{k+1} (j)}(T) = \alpha_{\sigma\sigma^k(j)}(T)= 
\alpha_{\sigma^k(j)}(T)-\alpha_{\sigma^k(j)}(f)-\alpha_{\tau \sigma^{k+1}(j)}(f).$$
Therefore, by induction hypothesis $$\alpha_{\sigma^{k+1} (j)}(T) =\alpha_{j}(T) - \sum_{p=0}^{k-1} \alpha_{\sigma^p(j)} (f) -  \sum_{p=1}^{k} \alpha_{\tau (\sigma^p(j))} (f)-\alpha_{\sigma^k(j)}(f)-\alpha_{\tau \sigma^{k+1}(j)}(f)$$ which leads to the required formula. 
\end{itemize}
\end{proof}

\subsection{Proof of Theorem  \ref{Th1}.}

Let $u\in D$ a set of distinguished representatives of  the $\langle\sigma_f, \sigma_T\rangle$-orbits. Let $S_u$ be the $\sigma_f$-cycle by $u$ and  $m$ be its length. We recall that  $u$ is   distinguished if either $\sigma_T(u) \notin S_u$ or $\sigma_T(u)\in \{u,\sigma_f(u)\}$. 
 
For clarity, we will denote $\alpha_{i}(f)=\alpha(i), \  \ \   \sigma_f= \sigma,\ \  \alpha_{i}(T)=\beta(i)$  \ \ and  \ $ J=\displaystyle \bigcup_{\delta \in <\sigma, \tau>} I_{\delta(u)}$,  
 \vskip -0.4truecm 
 \noindent noting that $J$ is an $f$-invariant set.
 
\medskip

{\bf (1)} We first prove that if $f$ is strongly reversible by some $T$ with $\sigma_T=\tau$ then $$\sum_{j\in S_u} \alpha(j)+\sum_{j\in \tau S_u} \alpha(j) =0.$$

\smallskip

Summing for $j\in S_u$ the equalities (2) of Consequences of Lemma 3.2, we get that $$ \sum_{j\in S_u} \beta(\sigma(j)) = \sum_{j\in S_u} \beta(j) - \sum_{j\in S_u} \alpha(j)- \sum_{j\in S_u} \alpha(\tau\sigma(j))$$ and since  $\displaystyle \sum_{j\in S_u} \beta(\sigma(j)) = \sum_{j\in S_u} \beta(j)$, we get  $$ \sum_{j\in S_u} \alpha( j)+ \sum_{j\in S_u} \alpha(\tau\sigma(j))=0.$$

\bigskip

{\bf (2)} We now prove that under the previous condition, the choice of $\alpha_u(T)$ satisfying an  eventual  additional condition determines a unique involution $T$ that reverses $f$ on $J$.

\medskip

{\bf CASE A : $S_u$ is $\tau$-invariant.}

\medskip

We first note that $\beta(u)= -\beta(u)$ if $\tau(u)=u$  and
$2\beta (u)= 2\alpha (u)$ if $\tau(u)=\sigma(u)$ are necessary conditions for reversibility by (3) of Consequences of Lemma 3.2.

\smallskip

We fix $\beta(u)$ such that $2\beta (u)= 0$  if $\tau(u)=u$ and such that $2\beta (u)= 2\alpha (u)$ if $\tau(u)=\sigma(u)$. We consider the map  $T$ defined by   $\sigma_T=\tau$, $\alpha_{u}(T)=\beta(u)$ and Conditions $(4_k)$ that is  $\displaystyle  \alpha_{u^k}(T)=\beta(u^k) = \beta(u) - \sum_{j=0}^{k-1} \alpha (u^j) -  \sum_{j=1}^{k} \alpha(\tau(u^j))$ for $0\leq k<m$, where $u^k =\sigma^k(u)$.

The map $T$ is therefore well defined as a bijection of $J$.

Supposing that  the condition $(\dagger \dagger)$ \  $\displaystyle  2\sum_{j\in S_u} \alpha(j) =0$ \ holds, we will prove that  $T$  is an involution that reverses the restriction to $J$ of $f$, by checking that $T$ satisfies Conditions (1) and (2) of  Consequences of Lemma 3.2.

\bigskip

$\bullet$ For (2), we have to check that  $\beta(u^{k+1})= \beta(u^{k})-\alpha (u^{k})-\alpha(\tau(u^{k+1}))$, for $0\leq k<m$.

By definition, $\displaystyle \beta(u^{k+1})=\beta({u}) - \sum_{j=0}^{k} \alpha(u^j) -  \sum_{j=1}^{k+1} \alpha(\tau (u^j))$ therefore 

$\displaystyle \beta(u^{k+1})=\left(\beta(u) - \sum_{j=0}^{k-1} \alpha(u^j) -  \sum_{j=1}^{k} \alpha(\tau (u^j))\right) -\alpha(u^k)- \alpha (\tau (u^{k+1})).$

The condition $(\dagger \dagger)$ is used for treating the case $k=m-1$.
\bigskip

$\bullet$ For (1), proof depends on the nature of $u$.

\medskip

Note that for all $j\in \{ 1,...,m \}$ it holds that $\tau (u^j)= u^{m-j}$ if $\tau(u)=u$ or
$\tau (u^j)= u^{m+1-j}$, if $\tau(u)=\sigma(u)$.

\bigskip

\ \  Suppose $\tau(u)=u$,  checking (1) is proving that $\displaystyle  \beta(u^{m-k})= -\beta({u^k})$.

\medskip

$\displaystyle  \beta(u^{m-k})=\beta(u) - \sum_{j=0}^{m-k-1} \alpha(u^j) -  \sum_{j=1}^{m-k} \alpha(\tau (u^j))$,

adding $0= \displaystyle \sum_{j=0}^{m-1} \alpha (u^j) +  \sum_{j=1}^{m} \alpha (\tau (u^j))$, one get:

$\displaystyle  \beta(u^{m-k})=\beta(u) + \sum_{j=m-k}^{m-1} \alpha(u^j) + \sum_{j=m-k+1}^{m} \alpha(\tau (u^j)))$,

\medskip

changing $j$ for $p=m-j$, and noting that $\beta(u)=-\beta(u)$, we have:

$\displaystyle  \beta(u^{m-k})=\beta(u) + \sum_{p=1}^{k} \alpha(u^{m-j}) + \sum_{p=0}^{k-1} \alpha(\tau (u^{m-j}))=-\beta(u) + \sum_{p=1}^{k} \alpha(\tau(u^{j})) + \sum_{p=0}^{k-1} \alpha(u^j)$.

\bigskip

\ \  Suppose that $\tau(u)=\sigma(u)$, checking (1) is proving that $\displaystyle  \beta( u^{m+1-k})= -\beta(u^k)$.

\medskip

$\displaystyle  \beta(u^{m+1-k})=\beta(u) - \sum_{j=0}^{m-k} \alpha(u^j) -  \sum_{j=1}^{m+1-k} \alpha(\tau (u^j))$.

Adding $0= \displaystyle \sum_{j=0}^{m-1} \alpha(u^j) +  \sum_{j=1}^{m} \alpha(\tau (u^j))$, one get

$\displaystyle  \beta(u^{m+1-k})=\beta(u) + \sum_{j=m+1-k}^{m-1} \alpha(u^j)  + \sum_{j=m-k+2}^{m} \alpha(\tau (u^j))$.

Changing $j$ for $p=m+1-j$,

$\displaystyle  \beta(u^{m+1-k})=\beta(u)+ \sum_{p=2}^{k} \alpha(u^{m+1-j}) + \sum_{p=1}^{k-1} \alpha(\tau (u^{m+1-j} ))$,

As $\tau (u^j)= u^{m+1-j}$ and  $\beta(u)=-\beta(u)+\alpha(u)+\alpha(\tau\sigma(u))$, we have

$ \displaystyle
\beta(u^{m+1-k})=-\beta(u)+\alpha(u)+\alpha(\tau(u^1)) + \sum_{p=2}^{k} \alpha(\tau(u^j)) + \sum_{p=1}^{k-1} \alpha(u^j )=- \beta(u^{k})$.

\bigskip

\bigskip

{\bf CASE B : $S_u$ and $\tau S_u$ are disjoint cycles.}

\medskip

We fix $\beta(u)$ arbitrary and we define a map $T$ verifying  $\sigma_T=\tau$, $\alpha_{u}(T)=\beta(u)$ and Conditions $\displaystyle (4_k) \ \  \beta(u^{k})=\beta({u}) - \sum_{j=0}^{k-1} \alpha(u^j) -  \sum_{j=1}^{k} \alpha(\tau (u^j))$  and   $(1_k) \ \ \beta(\tau(u^k))= -\beta(u^k)$  for $0\leq k<m$.

Supposing that  $2\sum_{j\in S_u} \alpha(j) =0$, we check that $T$ defined   by these conditions
 is an involution that reverses the restriction to $J$ of $f$, that is, $T$ satisfies conditions (1) and (2) of  Consequences of Lemma 3.2.

\medskip

$\bullet$  $(1)$  is  condition $(1_k)$.  
\medskip

$\bullet$ For $(2)$, we have to check that  $\beta(u^{k+1})= \beta(u^{k})-\alpha (u^{k})-\alpha(\tau(u^{k+1}))$. As in case A, we get: 

\noindent $\displaystyle \beta(u^{k+1})=\beta({u}) - \sum_{j=0}^{k} \alpha(u^j) -  \sum_{j=1}^{k+1} \alpha(\tau (u^j))=\beta(u) - \sum_{j=0}^{k-1} \alpha(u^j) -  \sum_{j=1}^{k} \alpha(\tau (u^j)) -\alpha(u^k)- \alpha (\tau (u^{k+1})).$



\section{Proof of Corollaries of Theorem \ref{Th1}}

\subsection{Proof of Corollary \ref{FO}}

\begin{enumerate}

\item Suppose  that $f$ is a strongly reversible element  with $\sigma_f$  a $n$-cycle (a cycle of length $n$).

By Properties \ref{Pro3} (1), $\displaystyle \alpha(f^p)= \sum_{j=0}^{p-1} \sigma_f^j(\alpha(f))$. Therefore  

$\displaystyle  2\alpha_i(f^n)= 2\sum_{j=0}^{n-1} \alpha_{i^j}(f)=A(f)=0$  for all $i\in \{1,...,n\}$ by Theorem \ref{Th1}.

Finally,  $\sigma_{f^{n}} =Id$  and  $2\alpha_i(f^{n})=0$ for all $i\in \{1,...,n\}$. Hence,   $f$  has finite order at most $2n$.

\medskip

This extends to the case where  $\sigma_f$ is a product of cycles with distinct lengths, noting that the restriction of $f$ to any cycle of $\sigma_f$ must be "separately" strongly reversible.

\medskip

The map $f=(\alpha_1, Id)\in G_1$ with $\alpha_1=\frac{p}{q}$ with $q>2$ has finite order but  is not strongly reversible in $G_1$.
 
\medskip

\item In $G_2$, the map $f=((\alpha_1, \alpha_2), (1,2))$ with $\alpha_2=-\alpha_1  \notin \mathbb Q$ has infinite order and is  strongly reversible.

\noindent Minimal strongly reversible elements do not exist since for a minimal element $f$, $\sigma_f $ is a $n$-cycle.
\end{enumerate}

\subsection{Composites of involutions in $G_n$. Proof of Corollary \ref{PI}}\ 

\textit{Items (1) and (2).} As involutions satisfy $A(f)=0$ and $A$ is a morphism, it holds that $I_n\subset Ker A$. 

\smallskip

Conversely, let $f\in G_n$ such that  $A(f)=2\sum \alpha_i(f)=0$. 

Let $\gamma$ be a $n$-cycle, $\tau = \gamma \sigma_{f}^{-1}$ and  $T$ be  the element of $G_n$ defined by  $\sigma_T=\tau$ and $\alpha_T =0$.  

Therefore $\sigma_{Tf}= \tau\sigma_{f}=  \gamma$ is a $n$-cycle and $A(Tf)=0$ since  $A(f)=A(T)=0$ and $A$ is a morphism. Thus by Remark \ref{revcyc} and Theorem \ref{Th1}, $Tf$ and $T$ are strongly reversible so both are product of 2 involutions. 
 
Consequently $f$ is a product of at most $4$ involutions. In particular, $Ker (A) \subseteq I_n$.

\smallskip

In addition, according to Corollary \ref{FO}, $Tf$ is also periodic and $f$  is a product of at most $2$ periodic IETs.

\smallskip
\textit{Item (3)- General case :  for $n\geq 5$, there exists $f\in G_n$ which can not be written as a product of $3$ involutions.}
\begin{defi} The \textbf{rank} of $f\in G_n$  is the rank of the subgroup of $\mathbb S_n$ generated by the $\alpha_i(f)$'s and $\frac{1}{2n}$.
\end{defi}

Let $ f$ be the element of $G_n$ defined by $\sigma_f=Id$ and $\alpha_f= ( \gamma_1, \cdots, \gamma_{n-1},\delta)$, where the $\gamma_i$'s are rationally independent irrational numbers and $\delta = -(\gamma_1 + \cdots +\gamma_{n-1})$. 

Then $A(f)=0$, $f$ is not strongly reversible and it has rank $n$, we claim that $f$ can not be written as a product of $3$ involutions.

We argue by contradiction supposing that $f= T r$ where $T$ is an involution and $r$ is strongly reversible by $l=(\tau, \alpha_l) \in G_n$ . In particular, one has $\sigma_r=\sigma_{T}$ is an involution and we decompose it as a product of 1 and 2 cycles of disjoint supports: 
$$\sigma_r= \prod_{i=1}^p (a_i,b_i)  \ \prod_{i=p+1}^{p+s } (a_i,b_i)  \ \prod_{i=p+s+1}^{p+s+t } (a_i) \ \prod_{i=p+s+t+1}^{p+s+t+v } (a_i) ,$$
where  $\{a_i, i=1,\cdots p+s+t+v\} \cup\{b_i, i=1, \cdots, p+s \} = \{1,...,n\}$, in particular $2p+2s+t+v= n$ and
\begin{enumerate}
\item $\tau(a_i,b_i)\not=(a_i,b_i)$ (it is another  2-cycle disjoint from $(a_i,b_i)$),  for $i=1,\cdots p$,
\item $\tau(a_i,b_i)=(a_i,b_i)$ for $i=p+1,\cdots p+s$,
\item $\tau(a_i)\not=a_i$ (it is another 1-cycle disjoint from $(a_i)$),  for $i=p+s+1,\cdots p+s+t$,
\item $\tau(a_i)=a_i$ for $i=p+s+t+1,\cdots p+s+t+v$.
\end{enumerate}

For clarity, for $g\in G_n$, we denote $\alpha_{a_i} (g) = \alpha_i(g)$ and $\alpha_{b_i} (g) =\beta_i(g)$. 

\medskip

According to Theorem 1, the reversibility of $r$ leads to the following equations:
\begin{enumerate}
\item $\alpha_i(r)+\beta_i(r)+\alpha_{\tau(i)}(r)+\beta_{\tau(i)}(r)=0$, for $i=1,\cdots p$,
\item $2(\alpha_i(r)+\beta_i(r))=0$ for $i=p+1,\cdots p+s$,
\item $\alpha_i(r)+\alpha_{\tau(i)}(r)=0$,  for $i=p+s+1,\cdots p+s+t$,
\item $2\alpha_i(r)=0$ for $i=p+s+t+1,\cdots p+s+t+v$.
\end{enumerate}

We get then $\frac{p}{2}+ s +\frac{t}{2}+ v$ independent equations.

\smallskip

In addition, since $T$ is an involution, $\alpha_T$ satisfies:

\begin{enumerate}
\item $\alpha_i(T)+\beta_{i}(T)=0$, for $i=1,\cdots ,p+s$,
\item $2\alpha_i(T)=0$ for $i\geq p+s+1$.
\end{enumerate}

\medskip

By properties \ref{Pro3}, $\alpha_{Tr} =\alpha_r + \sigma_r (\alpha_T) =\alpha_r + \sigma_T(\alpha_T)=\alpha_r -\alpha_T$.


It is easy to check that $\alpha_i(Tr)$ satisfies the equations (1), (2) and (4) of  $\alpha_i(r)$.

For equations (3) we get $\alpha_i(Tr)+\alpha_{\tau(i)}(Tr)=\alpha_i(r) -\alpha_i(T)+\alpha_{\tau(i)}(r) -\alpha_{\tau(i)}(T)=\alpha_i(r)+\alpha_{\tau(i)}(r) -(\alpha_i(T)+\alpha_{\tau(i)}(T))= -(\alpha_i(T)+\alpha_{\tau(i)}(T))$ by condition (3) for $r$. But  condition (2) for $\alpha_i(T)$ states that 
$2\alpha_i(T)=2\alpha_{\tau(i)}(T)=0$, so $2 (\alpha_i(T)+\alpha_{\tau(i)}(T))=0$. 

\smallskip

So $\alpha_i(Tr)$ satisfies  $\frac{p}{2}+ s +\frac{t}{2}+ v$ equations.

Therefore the rank of $Tr$ is at most $n+1- (\frac{p}{2}+ s +\frac{t}{2}+ v)\leq n+1-\frac{n}{4}$, the maximum being  for $2p=n, s=t=v=0$. This contradicts that $Tr$ has rank $n$, provided that $n\geq 5$.

\medskip

\textit{Item (3)- Remaining cases.}

In $G_1$, the subgroup $I_1=\{ Id, R_{\frac{1}{2}} \}$. In $G_2$, the subgroup $I_2$ consists in reversible maps. 

In $G_3$, the previous argument works since $p=0$ and therefore the maximal rank for a product of $3$ involutions is $2$.

In $G_4$, the maximal rank is obtained for $p=2, s=t=v=0$ and it is $4$. Let $f=Tr$, with $\sigma_T=\sigma_r = (a_1,b_1)(a_2,b_2)$ and $r$ is reversible by an involution that exchanges its 2-cycles. It holds that $\frac{1}{2} A(r) =\alpha_1(r)+\beta_1(r)+\alpha_{2}(r)+\beta_{2}(r)=0$ (condition (1) for $r$) and $\frac{1}{2} A(T) =(\alpha_1(T)+\beta_1(T))+ (\alpha_{2}(T)+\beta_{2}(T))=0$ (condition (1) for $T$). Therefore $\frac{1}{2} A(Tr) =\frac{1}{2} A(r) + \frac{1}{2} A(T) =0$.

Let $f\in G_4$. By Item (2), $f\in I_4$ if and only if $A(f) =0 \ (mod \frac{1}{4})$ that is $\frac{1}{2} A(f) \in\{0,\frac{1}{8}\} \  (mod \frac{1}{4})$. Then, any element $f$ of $I_4$ such that $\frac{1}{2} A(Tr) =\frac{1}{8}$ can not be written as a product of $3$ involutions.


\section{Reversibility in $G_n$. Proof of Theorem  \ref{Th2}.}

Let $f=(\alpha_f, \sigma_f) \in G_n$ be reversible by $T=(\alpha_T, \sigma_T) \in G_n$. We denote by $m$ the order of $\sigma_T$.

W.l.o.g. we can suppose that the action of  $\langle\sigma_T,\sigma_f\rangle$ on $\{1,\cdots, n\}$ is transitive.

\begin{lemm}\label{case}
Let $f\in G_n$ reversible by $T$ in $G_n$. Then $\sigma_f$ is $\sigma_T$-reversible in $\mathcal S_n$ and  if $i\in \{1,...,n\}$ belongs to a $\sigma_f$-cycle  of  length $p$ then 
$\sigma_T^s(i)$ generates a  $\sigma_f$-cycle of length $p$ and either
\begin{enumerate} 
\item $m$ is odd or 
\item $m$ is even and either 
\begin{enumerate}
\item there exists $s$ odd such that $\sigma_T^s(i) \in \langle\sigma_f\rangle(i)$
or
\item The map given for $j \in  \langle\sigma_f, \sigma_T\rangle(i)$  by 

\ \ $\bullet $ $ \epsilon(j)= 1$ \  \ if exists $s$ even such that $\sigma^s_T(j) \in \langle\sigma_f\rangle(i)$ 

\ \  $\bullet $ $ \epsilon(j)= -1$ if exists $s$ odd such that $\sigma^s_T(j) \in \langle\sigma_f\rangle(i)$,

\noindent is well defined.
\end{enumerate}
\end{enumerate}
\end{lemm}

\begin{proof}
By arguments in proof of Lemma \ref{cyc}, $\sigma_T^s(i)$ generates a  $\sigma_f$-cycle of  length $p$ and the only point to prove is (2).

Suppose that $m$ is even and there doesn't exist $s$ odd such that $\sigma_T^s(i) \in \langle\sigma_f\rangle(i)$.

First, we note that  by transitivity hypothesis and reversibility relation 
  $\sigma_f \sigma_T=\sigma_T \sigma_f ^{-1}$, given $j $, there exist $ s, t$ such that $j= \sigma_T^s \sigma_f^t(i)$.

The reason why $\epsilon(j)$ fails to be well defined is that there exist $ s, s'$ with $s-s'$ odd such that $\sigma_T^s \sigma_f^t(i) =\sigma_T^{s'} \sigma_f^{t'}(i) $ that is  $\sigma_T^{s-s'} \sigma_f^t(i) = \sigma_f^{t'}(i) $ and then 
$\sigma_T^{s-s'} (i)= \sigma_f^{t+t'}(i) $ which is a contradiction.\end{proof}

\medskip

\noindent{\bf Proof of Theorem \ref{Th2}  when $m$, the order of  $\sigma_T$, is odd} is given by the following

\begin{prop} \label{CP}
Let $f\in G_n$ reversible by $T$ and suppose that there exists $s$ odd such that $\sigma_T ^{s} \in \langle\sigma_f\rangle$ (this is weaker than requiring $\sigma_T$ has odd order) then
 
 \begin{enumerate}
 \item  $\sigma_f$ is an involution,
 \item  $f$ is strongly reversible and 
 \item $f$ has finite order $2$ or $4$.
 \end{enumerate}
 \end{prop}

\begin{proof}
\textit{Reduction to the case  $\sigma_T=Id$.} Let $f\in G_n$ reversible by $T$ with $\sigma_T ^{s} = \sigma_f^t$, with $s$ odd. By  Properties \ref{Pro2}, the map $T_1=T ^{s} f^{-t}$ reverses $f$ and $\sigma_{T_1}=Id$. 

\medskip

\ \ \ \  \  \textit{Let $f\in G_n$ reversible by $T$ with $\sigma_T=Id$.} 

\begin{enumerate}
\item  As $\sigma_T=Id$, one has  $\sigma_f =(\sigma_f)^{-1}$.
\item  The necessary and sufficient condition of Properties \ref{Pro3} (4) can be written as:
  $$\alpha_{T^{-1}}+ \alpha_f +\sigma_{f} (\alpha_T) + \sigma_f(\alpha_f)= 0,$$
composing by $\sigma_f$, one get 
 $$\sigma_{f}(\alpha_{T^{-1}})+ \sigma_{f}(\alpha_f) +\alpha_T + \alpha_f= 0.$$
As $\sigma_T=Id$ one has $\alpha_{T^{-1}}=-\alpha_T$, summing the previous equalities one get
 \[\label{CNS}  2(\alpha_f  + \sigma_f(\alpha_f))= 0.  \tag{A}\]

Since $\sigma_T=Id$ and $\sigma_f$ is an involution, Formula (\ref{CNS}) leads to the necessary and sufficient condition for strongly reversibility given by Theorem \ref{Th1}.

\smallskip

\item Formula (1) of Properties \ref{Pro3}  implies that  $\alpha_{f^2} = \alpha_f  + \sigma_f(\alpha_f)$   and   $\alpha_{f^4} =2\bigl( \alpha_f + \sigma_f(\alpha_f)\bigl)$. 

In conclusion, on has $\alpha_{f^4} =0$ by Formula(\ref{CNS})  and   $\sigma_f ^4= Id$, so $f^4=Id$. 
\end{enumerate} \end{proof} 

\noindent{\bf Proof of Theorem \ref{Th2} when $m$  is even}. Let $i \in \{1,....,n\}$.

\medskip

{\bf Case 1.} \textit{It doesn't  exist $s$ odd such that $\sigma_T^{s}(i) =\sigma_f^t (i)$.}

\smallskip

Item (2) of Lemma \ref{case} allows us to define $T_0$ as $T_0(x)= T^{\epsilon(j)}(x)$ for $x \in I_j$. 

Writing $j=\sigma_T^s \sigma_f^t (i)$ we have $T_0(x)= 
\left\{ \begin{array}{c} 
T(x) {\text{ \ \ \ if }} x\in I_{\sigma_T^s \sigma_f^t (i)} {\text{ and }}  s {\text{  is even, }} \\
T^{-1}(x) {\text{ if }} x\in I_{\sigma_T^s\sigma_f^t (i)} {\text{ and }} s {\text{ is \ odd.}} \end{array} \right.$

One can easily  check  that $T_0$ is an involution that reverses $f$.
\bigskip

{\bf Case 2.} \textit{There exists $s$ odd such that $\sigma_T^{s}(i) =\sigma_f^t (i)$.} Let $T_1= T^s f^{t}$.

\smallskip

By Properties \ref{Pro2}, the map $T_1$ reverses $f$ and $\sigma_{T_1}(i)=\sigma_T^s\sigma_f^t (i)= \sigma_f^{-t} \sigma_T^s(i)=i$. Then $\langle\sigma_{T_1},\sigma_f\rangle (i)=\langle\sigma_f \rangle(i)$.

In addition, we have $\sigma_{T_1}(\sigma_f^k (i))=\sigma_f^{-k} (\sigma_{T_1} (i))=\sigma_f^{-k} (i)$ and   $\sigma_{T_1}(\sigma_f^{-k} (i))=\sigma_f^{k} (i)$,  for any $k$. 

In conclusion, $\sigma_{T_1}$ is an involution on $\langle\sigma_f\rangle (i)$ and $\sigma_{T_1}(i)=i$. 

\smallskip

W.l.o.g. we can consider $\sigma_f= (1,....,n)$ and $ \sigma_{T_1}(j)=N-j $ for some $N=1,..., n+1$.

For clarity, we denote $\alpha_j(f)=\alpha_j$ and $\alpha_j(T_1)=\beta_j$, for all $j=1,..., n$. Let $x \in I_j$, 

\noindent one has that $T_1\circ f(x)=x+\alpha_j+\beta_{j+1}$ and $f^{-1}\circ T_1(x)= x+\beta_j-\alpha_{N-1-j}$.

\medskip

Therefore $T_1$ reverses $f$ if and only if for any $j=1,...n$ it holds that $$\alpha_j+\beta_{j+1}=\beta_j-\alpha_{N-1-j}.$$
Summing from 1 to $n$ we get $2 \sum_{j=1}^n \alpha_j=0$. We conclude by noting that it is the necessary and sufficient condition for strongly reversibility in the case of a $\sigma_{T_1}$-invariant $\sigma_f$-cycle.


\section{Reversibility in $ \mathcal G$. Proofs of Theorems \ref{Th3}, \ref{Th4.1}, \ref{Th4} and \ref{extfree}.}
\subsection{Preliminaries.}
\subsubsection{Dynamical properties of IETs}

In this section, we recall the Arnoux-Keane-Mayer decomposition Theorem  (\cite{Ar}, \cite{Ke}, \cite{May}) and Novak's work on the growth rate of the number of discontinuities for iterates of an IET (\cite{No}).

\begin{defi} Let $f\in \mathcal G$.

The \textbf{break point set} of $f$ is  obtained  by adding  the initial point $\{0\}$ to the discontinuity set of $f$,  it is denoted by $BP(f)$.

The set consisting in the $f$-orbits of points in  $BP(f)$ is  denoted by $BP_{\infty}(f)$.

A subset $V$ of $[0,1)$ is said \textbf{of type $\mathcal M$} if it is a non empty finite union of intervals each of the form $[b,c) $, with $b,c$ in $BP_{\infty}(f)$.  A type $\mathcal M$ and $f$-invariant  set  that is minimal for the inclusion among  type $\mathcal M$ and $f$-invariant subsets of $[0,1)$ is called an \textbf{$f$-component}.
\end{defi}

\smallskip
It is well known that IETs decompose into minimal and periodic components. This decomposition was first studied for measured surface flows by
Mayer in 1943  (\cite{May}) and restated for IETs by Arnoux (\cite{Ar}) and Keane (\cite{Ke}).

\smallskip

\textbf{The Arnoux-Keane-Mayer decomposition Theorem}  claims that $[0,1)$ can be decomposed as $[0,1)= P_1\cup ...P_l \cup M_1...\cup M_m,$  where
\begin{itemize}
\item $P_i$ is an  $f$-periodic component: $P_i$ is the $f$-orbit of an interval $[b,c)$ with $b,c$ in $BP_{\infty} (f)$ and all iterates $f^k$ of $f$ are continuous on $[b,c)$.   In particular points in $P_i$ are periodic of the same period. 
\item $M_j$ is an  $f$-minimal component: for any $x\in M_j$, the orbit $\mathcal O_f(x)$ is dense in $M_j$.
\end{itemize}

\begin{remar}\label{perp}
Note that $f$-periodic points of same period $p$ may belong to distinct components however the set $Per_p(f)$ consisting in $f$-periodic points of period $p$ is a finite union of periodic components and it is a type $\mathcal M$ and $f$-invariant  subset.
\end{remar}

\begin{propri}\label{conjrot} (\cite{No}, Lemma 5.1 and its proof.)
\item Two irrational rotations $R_\alpha$  and  $R_\beta$ with $\alpha\not= \beta (mod 1)$ are nonconjugate in $\mathcal G$. 
\item The centralizer in $\mathcal G$ of an irrational rotation is the rotation group $\mathbb S^1$.
\end{propri}

\begin{lemm} \label{comp} \ 
Let $f$ and $h$ be  IETs such that $h$ reverses $f$.
\begin{itemize}
\item The image by $h$ of a minimal component of $f$  is a minimal component of $f$.
\item Given $p\in \mathbb N^*$, the set $Per_p(f)$  is $h$-invariant.
\end{itemize}
\end{lemm}

Indeed, let $x \in M$ a minimal component of $f$. Since $f^n(h(x))= h(f^{-n} (x))$ one has $\mathcal O_f(h(x))=h(\mathcal O_f(x))$ and therefore  $\mathcal O_f(h(x))$ is dense in $h(M)$ meaning that $h(x)$ belong to a minimal $f$-component that is exactly $h(M)$.

Let $x$ be a $p$-periodic point of $f$, one has $f^p(h(x))= h(f^{-p}(x)) =h(x)$ and 
$f^k(h(x)) = h(f^{-k}(x)) \not=h(x)$ for $0<k<p$. Therefore $h(x)$ is $f$-periodic of period $p$.  \hfill $\square$

\smallskip

We let the reader check the following
\begin{propri}\label{proBP} Let $f$ and $g$ in $\mathcal G$. 
\begin{enumerate}[(a)]  
\item $BP(f\circ g) \subseteq BP ( g)\cup g^{-1}(BP (f))$,
\item  $BP (f^{-1})=f(BP (f))$ and
\item $BP (f^n) \subseteq BP (f)\cup f^{-1}(BP (f))\cup ...\cup  f^{-n+1}(BP (f)).$
\end{enumerate}
\end{propri}

\subsubsection{Basic algebra of $BS(1,-1)=\langle a, b  \ \vert \ bab^{-1}=a^{-1} \rangle$} \ 
\begin{itemize}
\item The index $2$ subgroup $\langle a, b^2 \rangle$  is isomorphic to $\mathbb Z^2$.
\item The index $2$ subgroup $\langle b, a^2 \rangle$ is isomorphic to $BS(1,-1$). 
\item For any action of $BS(1,-1)$, the element $b$ preserves the fixed points set of $a$.
\end{itemize}

\begin{lemm}\label{FNBS} \  
\begin{enumerate}
\item Every element of $BS(1,-1)$ is equal to a unique element $a^pb^q$ with
$p$ and $q$ integers. 
\item An action $\rho : BS(1,-1) \longrightarrow \text{Bij}(X)$  on a space $X$ is not faithful if an only if there exists a positive integer $p$ such that either $\rho(a)^p= Id$ or $\rho(b)^{2p}= Id$. 
\end{enumerate}
\end{lemm}
\begin{proof}
Item (1).  Existence is easy and uniqueness can be proved by considering some specific
faithful action of $BS(1,-1)$, for instance the one described on the first figure of this paper.

\smallskip

Item (2).  Let us denote $\alpha= \rho(a)$ and $\beta=\rho(b)$. According to the normal form described in item (1), if $\rho$ is not faithful there exist two integers $p$ and $q$  such that $(p,q)\not=(0,0)$ and $\alpha^p \beta^q= Id$.

\quad If $q$ is odd then $Id=\alpha^p \beta^q$ reverses $\alpha$ which implies that $\alpha$ is an involution. 

\quad If $q=0$ then $\alpha^p =Id$.

\quad If $q\not=0$ is even then if $p=0$ then $\beta^q=Id$ if $p\neq 0$, $\alpha^p =\beta^{-q}$ commutes with $\beta$ but $\beta$ also reverses $\alpha^p$ so  $\alpha^p$ is an involution. \end{proof}

\subsection{Free faithful actions of $BS(1,-1)$ by IET. Proof of Theorem \ref{Th3}.}

\begin{defi} We define the growth rate of the number of discontinuities for the iterates of an IET $f$ on the $f$-orbit through a given point $x$ by
$${\mathcal N}_x(f)= \lim \limits_{n \rightarrow + \infty} \frac{\# \{ BP(f^n) \cap {\mathcal O}_x(f)\}}{n}.$$

\end{defi}
\begin{propri}\label{inva} Let $f$, $g$ in $\mathcal G$ and $x\in [0,1)$. 
\begin{enumerate}
\item ${\mathcal N}_x(f)$ is well defined and ${\mathcal N}_x(f)\in \{0, 1\}$.
\item  If ${\mathcal N}_x(f)=0$ for any $x$ then there exists $p$ such that $f^p$ is conjugated in $\mathcal G$ to a product of restricted rotations of pairwise disjoint supports.
\item ${\mathcal N}_x(f)={\mathcal N}_x(f^{-1}).$
\item $ BP(f) \cap {\mathcal O}_x(f)=\emptyset$ then  ${\mathcal N}_x(f)=0$.
\item $ {\mathcal N}_{g(x)}(g \circ f \circ g^{-1})={\mathcal N}_x(f).$
\end{enumerate}
\end{propri}
\begin{proof} Items (1) and (3) are consequences of Novak's work in Section 2 of \cite{No}.

\smallskip
Item (2) is a reformulation of Theorem 1.2 of \cite{No}.

\smallskip
Item (4) follows  from  item (c) of Properties  \ref{proBP}.

\smallskip
To prove item (5), we use item  (a) of Properties \ref{proBP}:

\smallskip

$BP(g \circ f^n \circ g^{-1})\subseteq BP (g^{-1})\cup g(BP (f^n)) \cup g\circ f^{-n}(BP (g))$ and setting $f_g=g \circ f \circ g^{-1}$,

\smallskip
 
$BP( f^n )=BP(g^{-1}(g \circ f^n \circ g^{-1})g) \subseteq BP (g)\cup g^{-1}(BP (f_g^n)) \cup g^{-1} \circ f_g^{-n}(BP (g^{-1}))$.
  
\medskip
  
Noting that $y\in  {\mathcal O}_x(f) $ if and only if  $g(y)\in  {\mathcal O}_{g(x)}(f_g)$ and setting $C=  2 \# BP(g)$, we get 
$$ \# \left( BP(f^n) \cap {\mathcal O}_{x}(f) \right)- C \leq  
\# \left( BP(f_g^n) \cap {\mathcal O}_{g(x)}(f_g)\right) \leq \# \left( BP(f^n) \cap {\mathcal O}_{x} (f) \right)+ C.$$
We conclude by dividing by $n$ and taking the limit. \end{proof}

{\bf Proof of Theorem \ref{Th3}} is inspired from the proof given by Minakawa for describing free actions of $\mathbb Z^2$ by circle PL-homeomorphisms (see \cite{Mi}).

The action $\langle f,h\rangle$ of $BS(1,-1)$ is free thus for every $x\in I$, one has ${\mathcal O}_{h^k(x)}(f) \cap {\mathcal O}_{h^q(x)}(f) = \emptyset$ provided that $k \neq q$.

Therefore, as $\# BP(f)$ is finite, there exists $N_0$ such that  ${\mathcal O}_{h^n(x)}(f) \cap BP(f) = \emptyset$ for any $n\geq N_0$.

Using item (4) of Properties \ref{inva}, one get  ${\mathcal N}_{h^n(x)}(f)=0$ for  $n\geq N_0$, in addition by item (5) $$ {\mathcal N}_{h^n(x)}(h^n \circ f \circ h^{-n})={\mathcal N}_x(f).$$

By reversibility, it holds that $ {\mathcal N}_{h^n(x)}(h^n \circ f \circ h^{-n})={\mathcal N}_{h^n(x)}(f^{\epsilon})$, where $\epsilon =1$ for $ n$ even and  $\epsilon =-1$ for $ n$ odd. Summarizing, for any $x \in I$, we get ${\mathcal N}_x(f)=0$.

Then by item (2) of Properties \ref{inva},  there exists $p$ such that $f^p$ is conjugated by $E$ in $\mathcal G$ to $\phi$, a product of restricted rotations of pairwise disjoint supports. Moreover, these restricted rotations are of infinite order and $I$ is the union of their supports since the action is faithful and free.

We apply to $\phi$ the following

\begin{lemm}\label{Lem7.2}
Let $\phi$ be a product of  restricted rotations of pairwise disjoint supports. Then there exist a positive integer $n$ and a PL-homeomorphism $R: I \rightarrow I$ such that $F=R\circ  \phi \circ R^{-1} \in G_n$ and $\sigma_F=Id$. Moreover, the map  $R^{-1}$ is affine on the $I_i$'s.
\end{lemm}
\begin{proof} Let us decompose $I$ as a disjoint union of consecutive half open intervals $J_i$, $i=1,...,n$, where $J_i$ is either the support of a restricted rotation or it is a connected component of  $I\setminus supp(F)$. 

We define $R$ as the PL-homeomorphism such that $R(J_i)=I_i=[ \frac{i-1}{n},  \frac{i}{n})$. It is easily seen that  $F=R\circ \phi  \circ R^{-1} \in G_n$ and $\sigma_F=Id$.\end{proof}

Therefore  $F=R\circ \phi  \circ R^{-1} \in G_n$, $\sigma_F=Id$ and $F_{\vert I_i}$ is minimal. We first prove that the corresponding $H=R\circ \eta \circ R^{-1} \in G_n$ where $\eta=E\circ h\circ E^{-1}$, that is a priori an AIET, also belongs to $G_n$. 

Let us denote  $J_i$, $ i=1,....,n$ the minimal components of $\phi$. According to Lemma \ref{comp}, there exists a permutation $\gamma \in \mathcal S_n$ such that
the IET $\eta$ sends $J_i$  to $J_{\gamma(i)}$. In particular $J_i$  and $J_{\gamma(i)}$ have the same length. Conjugating by the PL-homeomorphism $R$ we get that $H$ sends $R(J_i)=I_i$ to $R(J_{\gamma(i)})=I_{\gamma(i)}$.
Noting that the restrictions of $R$ to $J_i$ and to $J_{\gamma (i)}$ are affine with the same slope, we get that $H$ is an IET.

Since $H$ conjugates $F_{\vert I_i}$ to $F^{-1}_{\vert I_{\gamma(i)}}$,
Properties \ref{conjrot} implies that $\alpha_{\gamma(i)} (F) =-\alpha_{i}(F)$.

Therefore it can be check that $F=(\alpha_F, Id)$ is reversible by $g\in G_n$ defined by $\alpha (g) =0$ and $\sigma_g = \gamma$.

Finally, $H g^{-1}$ is an IET that commutes with $F$ and fixes each $I_i$. Since the centralizer of an irrational rotation is the rotation group $\mathbb S^1$, one has  $H g^{-1} = R_{\beta_i}$ on any $I_i$. Thus $H =g R_{\beta_i}$ on any $I_i$ meaning that  $H\in G_n$.

We finishes the proof of Theorem  \ref{Th3} by proving analogously that $(RE)\circ f \circ (RE)^{-1}\in G_n$. This is provided by the following 
 
 \begin{lemm}\label{VirtGn}
Let $f$ be an IET without periodic points and such that  $f^p\in G_n$  for some positive integer $p$. Then $f\in G_n$.
 \end{lemm}
 
Eventually passing to a power of $f^p$ we can suppose that  $\sigma_{f^p}=Id$.
 
First, notice that $f$ permutes the minimal components $I_i$ of $f^p$ that is there exists $\sigma \in \mathcal S_n$ such that $f(I_i) = I_{\sigma(i)}$.

Since $ f^p$ is an irrational rotation on $I_i$ and  $f$ commutes with $f^p$, it follows that $f$ sends an interval where $f^p$ is a rotation of angle  $\alpha$ to  one  with the same angle. This is a consequence of the fact that the angle $\alpha$  of a rotation is invariant by conjugacy in IET.
  
Thus, the periodic element  $\tau$ of  $G_n$ defined by $\sigma_\tau=\sigma$ and $\alpha(\tau) = 0$ commutes with $F^p$.  
 
Therefore $\tau^{-1} f$ preserves each $I_i$ and commutes with $f^p$.
  
Since the centralizer of an irrational rotation is the rotation group $\mathbb S^1$, one has $\tau ^{-1} f = R_{\beta_i}$ on any $I_i$. Thus $f = \tau R_{\beta_i}$ on any $I_i$ meaning that $f\in G_n$.
\medskip

\begin{remar}\label{Raffine}
Noting that $h$ and $f$ permute the minimal components of $f^p$, we get that any $M_i:=\langle\phi, \eta\rangle (J_i)$ is the union of finitely many $J_k$ that have same length and then $R$ is affine on    any $M_i$.
\end{remar}
\subsection{Some extensions to other groups. Proof of Theorem \ref{extfree}.}
By similar arguments, we can establish the following:

\begin{prop}\label{z2}
Any free faithful actions of $\mathbb Z^2$ by IET is ($PL\circ IET$)-conjugated to a  $\mathbb Z^2$-action in some $G_n$.
\end{prop}

\begin{prop}\label{CGn}
Let $f_1$, $f_2$ commuting IET satisfying:
\begin{enumerate}
\item $Per(f_1)=Per(f_2)= \emptyset$ and
\item $f_1$ is conjugated by $ A$ to $F_1$ in $G_n$, where $A= R\circ E$, with $E\in \mathcal G$ and $R$ a PL-homeomorphism such that $R^{-1}$ affine on the $I_i$'s.
\end{enumerate}
then $f_2$ is conjugated by $ A$ to $F_2$ in $G_n$.
\end{prop}

Indeed, eventually changing $f_1$ for an iterate, we can suppose that the $J_i = R^{-1} (I_i)$ are the minimal components of $f_1$. As $f_2$ permutes the $J_i$'s, then there exists a positive integer $m$  such that the map $F_2^m= A f_2^m A^{-1}$ is an IET that preserves each $I_i$ and commutes with $F_1$. By properties 6.1, $F_2^m\in G_n$ therefore by Lemma 6.4, $F_2\in G_n$.

\smallskip

As a consequences of these propositions we have
\begin{theo}\label{ext2}
Let $G$ be a finitely generated virtually abelian and torsion free group.
If one of the following properties is satisfied 
\begin{enumerate}
\item $G$ contains an element conjugated to a product of restricted rotations with disjoint supports and without periodic points.
\item there exists a subgroup $\Gamma$ of $G$ that is isomorphic to $\mathbb Z ^2$ and acts freely on $I$.
\end{enumerate}
Then $ G$ is conjugated in  $PL\circ IET$ to a subgroup of some $G_n$.
\end{theo}

\begin{remar}
As  a consequence of the previous result, we get Theorem \ref{extfree} since it deals with  groups that satisfy Item (2).
\end{remar}

\noindent {\bf Proof of Theorem \ref{ext2}.}

Let $\{ f_1, \cdots, f_r \}$ be a finite generating set of $G$. We claim that there exist positive integers $p_i$ for $i=1,\cdots, r$ such that $ f_i^{p_i}$ are pairwise commuting. Indeed, by hypothesis there exists $K$ an abelian normal subgroup of $G$ with finite index, $n$. Let $\{m_1, \cdots, m_n\}$ be a set of representatives of $ G/K$. Then, for any $i=1,\cdots, r$ and $s \in \mathbb Z$ there exists $j_s=j(i,s) $ such that $ f_i ^s  $ belongs to the class modulo $K$ of $m_{j_s}$.
By finiteness there exist $ s < t$ such that $ f_i ^s  $ and $ f_i ^t  $ are in the same class modulo $K$. That is, $f_i ^{t-s} \in K$ so taking $p_i=t-s$ we get the claim.

Suppose that $G$ satisfies the first item of Theorem \ref{ext2}, eventually increasing the generating set of $G$ we can suppose that $f_1$ is the element under consideration in this item. In particular, any iterate of $f_1$ is also a product of restricted rotations with disjoint supports and without periodic points. Therefore, by Lemma \ref{Lem7.2}, $f_1$ is conjugated by  $A$ in $PL$ to an element of $G_n$.

Suppose that $G$ satisfies the second item of the theorem, by Proposition \ref{z2}, the subgroup $\Gamma$ is  conjugated in  $PL\circ IET$ to a subgroup of some $G_n$. Eventually increasing the generating set of $G$ we can suppose that $f_1$ belogs to $\Gamma$ 

In both cases, we have that   $f_1^{p_1}$ is conjugated by  $A$ in $PL\circ IET $ to an element of $G_n$.

Let $i=2,\cdots, r$, we apply Proposition \ref{CGn} to $f_1^{p_1}$ and $f_i^{p_i}$ and we obtain that $f_i^{p_i}$   is conjugated by  $A$ to an element of $G_n$.

Finally, re-applying Proposition \ref{CGn} to $f_i^{p_i}$ and $f_i$ and we obtain that $f_i$   is conjugated by  $A$  to an element of $G_n$.\hfill $\square$

\subsection{Proof of Theorem \ref{Th4.1}.}
\begin{lemm}\label{FixhPerf} Let $f, h $ be  IETs.
\begin{itemize}
\item If $h$ commutes with $f$ minimal then either $h=Id$ or $Fix(h)=\emptyset.$
\item If  $h$ reverses $f$ then $Fix(h)\subset Per(f)$.
\end{itemize}
\end{lemm}

\begin{proof} Let $x\in Fix (h)$.

$\bullet$ If $h$ commutes with $f$, one has $h(f^{n}(x))=f^{n}(h(x))= f^{n} (x)$, then $\mathcal O_f(x) \subset Fix(h)$. By minimality, $Fix(h) = [0,1)$ and $h=Id$.

\smallskip

$\bullet$ If $h$ reverses $f$,  since $h$ is an IET there exists a subinterval $I_x=[x,c) \subset Fix(h)$. We now argue by contradiction supposing that $x \notin Per(f)$ therefore $\mathcal O_f(x)$ is locally  dense, in particular there exists a subsequence $p_n$ such that $f^{p_n}(x) \rightarrow x_+$. 
 
By reversibility,  $h(f^{p_n}(x))=f^{-p_n}(h(x))= f^{-p_n} (x)$ and as $f^{p_n} (x) \rightarrow x_+$, for $n$ sufficiently large $f^{p_n}(x)\in Fix(h)$ and then $f^{p_n}(x) = f^{-p_n} (x)$  meaning that $x$ is $f$-periodic.  \end{proof}

\noindent{\bf Proof of Theorem \ref{Th4.1}.}

(1)- Let $f$ be an minimal IET  reversible by $h$. We argue by contradiction supposing that
$\langle f,h \rangle$ is not free. Therefore by Lemma \ref{FNBS} (1),  there exist $x\in I$ and two integers $p$, $q$  such that $g= f^p h^q\not=Id$ and $g(x)= x$, this contradicts Lemma \ref{FixhPerf}. 

\medskip

(2) Suppose that $f$ is minimal and $h$ has infinite order, then by the previous item, $\langle f,h\rangle$ generates a faithful and free action. Thus, Theorems \ref{Th2}, \ref{Th3} and Corollary \ref{FO} imply that $\langle f,h\rangle$ is conjugated to an action in $G_n$ and $f$ can not be minimal. 

\medskip

(3) \textit{Periodic  IETs are strongly reversible.}

Let $f$ be a periodic IET, by the Arnoux-Keane-Mayer decomposition Theorem, $I=[0,1)$ can be written as the union of finitely many $f$-periodic components  $M_i$, $i=1,..,n$ of period $p_i$. In particular, $M_i=\sqcup_{k=1}^{p_i} J_k$, where $ J_k= f^{k-1}(J_1)$,  $J_1=[b,c[$ where $b,c \in BP_{\infty} (f)$ and $f^k$ is continuous on $J_1$.

Eventually conjugating $f$ and $h$ by an IET, we can suppose that the $J_k$'s are ordered consecutive intervals so the $M_i$'s are intervals.


Let us fix $i\in \{1,...,n\}$ and denote $m=p_i$.

Since  $ J_k= f^{k-1}(J_1)$,  all $J_k$, $k=1, ..., p_i$  have same length.  Therefore the map $H: M_i\rightarrow I$ defined by $H(J_k) = [\frac{k-1}{m},\frac{k}{m})$ is a homothecy that conjugates the restriction  $f_{\vert M_i}$  to an element of $G_m$ with $\alpha(H f_{\vert M_i} H^{-1})= 0$.

By Theorem 1, we conclude that $H f_{\vert M_i} H^{-1}$ is strongly reversible, so  there are two involutions $i_1$ and $i_2$ in $G_m$ such that $H f_{\vert M_i} H^{-1}=i_1 i_2$.

Finally $f_{\vert M_i} = (H^{-1} i_1 H)  ( H^{-1}i_2 H)$, where  $H^{-1} i_1 H$ and $H^{-1}i_2 H$ are involutions in $\mathcal G$ with supports in $M_i$. Therefore, for all $i$, $ f_{\vert M_i}$ is strongly reversible. This leads to the required statement.  \hfill $\square$

\subsection{Proof of Theorem \ref{Th4} and Corollary \ref{corpiper}.}

\begin{prop} \label{deco}\
If  $f$ is reversible in IET by $h$  then $I$ can be decomposed as $I=M_1\sqcup M_2 ... \sqcup M_l$, where $M_i$ are type $\mathcal M$ $f$-invariant subsets that are $h$-invariant  and  either

\begin{enumerate}
\item  There exists an integer $p_i$ such that any point of $M_i$ is $f$-periodic of period $p_i$;

\smallskip

In the following items, each $M_i$ is the union of finitely many $f$-minimal components and $h$ acts transitively on these components.

\item  the action of $\langle f,h \rangle$ on $M_i$ is a $BS(1,-1)$ faithful free action;

\smallskip
\item  the action of $\langle f,h \rangle$ on $M_i$ is a non  faithful $BS(1,-1)$ action: moreover there exists an even integer $p_i$ even such that the restriction $h^{p_i}_{\vert M_i}=Id$;

\smallskip
\item there exists an even integer $p_i$ such that $M_i$ is a union of $p_i$ $f$-minimal components.
\end{enumerate}
\end{prop}

\begin{proof} According to the Arnoux-Keane-Mayer decomposition Theorem, $I$ can be decomposed as a finite union of $f$-minimal or $f$-periodic components. Let $N$ be a component of $f$. 

\smallskip

If $N$ is $p$-periodic, we define $M=Per_p(f)$.

If $N$ is minimal, according to  Lemma \ref{comp} and since $f$ has finitely many components, there exists a least integer $s$ such that  $h^{s}(N)=N$ and  we define $ M=\bigcup_{j=0}^{s-1} h^j(N)$.

\smallskip

{\bf Case 1: }\textit{$N$ is periodic.} By Lemma \ref{comp} and Remark \ref{perp}, the set $\displaystyle M$ is of type $\mathcal M$, it is invariant by $f$ and $h$ and it satisfies item (1).

\medskip

{\bf Case 2: }\textit{$N$ is minimal and the action of $\langle f,h \rangle$ on $M$ is faithful and free.}
\medskip

{\bf Case 3: }\textit{$N$ is minimal and  the action of $\langle f,h \rangle$ on $M$ is either not faithful or not free.}

If the action is not faithful, as $f$ is minimal, by Lemma \ref{FNBS}, the map $h_{\vert N}$ has even order and the same holds for $h_{\vert M}$. Thus, we are exactly in the situation described by item (3).

\smallskip

If the action is not free, therefore by Lemma \ref{FNBS} there exist $x\in N$, $p=p_1s$  a multiple of $s$ and $q$ integers such that $h^p(x)=f^q(x)$.

Therefore $H=h^p f^{-q} _{\vert M} $ satisfies $H(x)=x$. If $p$ is odd then $H$ reverses $F=f_{\vert M}$, this contradicts Lemma \ref{FixhPerf}, so $p$ is even. Then $h^p(x)= f^{q}(x)$ and for all $n\in \mathbb Z$, one has $ h^p(f^n(x))= f^{n} (h^p(x))=f^n (f^q(x))= f^q (f^n(x))$. Thus $h^p_{\vert\mathcal O_f(x)}=f^q_{\vert \mathcal O_f(x)}$ and by minimality $h^p_{\vert N} =f^q_{\vert N}$.

If $q=0$ then $h^p_{\vert N}=Id$ and $h^p_{\vert M}=Id$. From now on, we consider $q\neq 0$.

If $h$ preserves $N$ then $f^q_{\vert N}$ commutes with $h$, this contradicts $h$ is a reverser of $f^q$.

If $h(N)\not=N$ and $s$ is odd then  $p_1$ is even, $H=h^s$ is a reverser of $f$ that preserves $N$ and $H^{p_1}(x)=f^q(x)$. We conclude that $f^q_{\vert N}$ commutes with $H$, this contradicts $H$ is a reverser of $f^q$.

\smallskip

If $h(N)\not=N$ and $s$ is even, we are in the situation described by item (4). 
\end{proof}

{\bf Proof of Theorem \ref{Th4}.} We will prove that $f_{\vert M_i}$ is reversible by a finite order element in any situation described by Proposition \ref{deco}.

{\bf 1. }\textit{$M$ is a finite union of $f$-periodic components of period $p$.} 


We conclude that $f$ is strongly reversible by  Theorem \ref{Th4.1} Item (3).

\medskip

{\bf 2. }\textit{The action of $\langle f,h\rangle$ on $M_i$ is a $BS(1,-1)$ faithful free action.} We conclude that $f$ is strongly reversible by Theorem \ref{Th3}, Theorem \ref{Th2} and the fact that the map $R$ is affine on $M_i$ according to Remark \ref{Raffine}.

\smallskip

{\bf 3. }\textit{There exists an even integer $p_i$ such that the restriction $h^{p_i}_{\vert M_i}=Id$.}  Therefore $p_i=2l$ and  either $l$ is odd and $h^l$ is an involution reversing $f$ or $l$ is even and  $h$ has order a multiple of $4$.

\smallskip

{\bf 4. }\textit{There exists an even integer $p_i$ such that $ M_i=\bigcup_{j=0}^{p_i-1} h^j(N)$ is a union of an even number of $f$-minimal components.} The map $h_0$ defined by 
$$ h_0(x)=  \left\{ \begin{array}{c} 
 h(x) {\text{ \ \ \ if }}  x\in h^j(N) {\text{ and }}  j {\text{  is even }} \\
h^{-1}(x) {\text{ if }} x\in   h^j(N) {\text{ and }} j {\text{ is \ odd,}} 
\end{array} \right.$$ is an involution that reverses $f\vert M_i$.

\bigskip
{\bf Proof of Corollary \ref{corpiper}.} \textit{Any reversible IET is a product of at most 2 periodic IETs.}
 
According to Theorem \ref{Th4}, a reversible IET $f$ is reversible by $h$ with even order $2p$. As $f=h (h^{-1}f)$, and $( h^{-1}f) ^{2p}=(h^{-1}f h^{-1}f)^p=(h^{-1}f f^{-1}h^{-1})^p= h^{-2p}=Id$, we have that $f$ is the product of 2 periodic IETs.

\section{Composites of involutions for 3-IETs and 4-IETs. Proof of Theorem \ref{Th7}.}

\subsection{The SAF invariant} It was introduced independently by Sah (\cite{Sah}) and Arnoux-Fathi (\cite{Ar}). We recall its definition and properties following Boshernitzan (\cite{Bos}).

\smallskip

Denote by $\mathbb R \otimes_{\mathbb Q} \mathbb R$ the tensor product of two copies of real numbers viewed as vector spaces over $\mathbb Q$: the space of finite sums of formal products $\lambda\otimes \gamma$ up to the equivalence properties $(\lambda+\lambda')\otimes \gamma=\lambda\otimes \gamma+\lambda'\otimes \gamma$, $\lambda\otimes (\gamma+\gamma')=\lambda\otimes \gamma + \lambda\otimes \gamma'$ and $q(\lambda\otimes \gamma)=(q\lambda)\otimes \gamma= \lambda\otimes (q\gamma)$, for $\lambda,\gamma$ in $\mathbb R$ and $q\in \mathbb Q$.

\smallskip

Denote by $\mathbb R \wedge_{\mathbb Q} \mathbb R$ the skew symmetric tensor product of two copies of reals: the vector subspace of $\mathbb R \otimes_{\mathbb Q} \mathbb R$  spanned by the wedge products $\lambda\wedge  \gamma := \lambda\otimes \gamma - \gamma  \otimes \lambda$, \ \  $\lambda,\gamma$ in $\mathbb R$.

\smallskip

\begin{defi}
The {\bf Sah-Arnoux-Fathi (SAF) invariant} is defined by $$\displaystyle SAF(f)=\sum_{k=1}^r  \lambda_k \otimes \gamma_k \in \mathbb R \otimes_{\mathbb Q} \mathbb R,$$ where the vectors $(\lambda_k)\in \mathbb R ^r$ encode the lengths of exchanged intervals and $(\gamma_k)$ the corresponding translation constants  respectively. 
\end{defi}

\smallskip

\begin{propri} \label{PropSAF} \ 
\begin{enumerate}
\item  $SAF: \mathcal{G} \rightarrow \mathbb R \otimes_{\mathbb Q} \mathbb R$ is a group homomorphism;
\item  For rotations, $SAF(R_{\beta})=  - 1\wedge \beta$, in particular $SAF(R_{\beta})=0$ if and only if $R_{\beta}$ is periodic.
\end{enumerate}
\end{propri}
As a consequence, any reversible IET has zero SAF-invariant.

\subsection{Proof of Theorem \ref{Th7}.}
\subsubsection{Item 1: 3-IETs with zero SAF-invariant are periodic.}
\begin{proof}
Let $f$ be a 3-IET having zero SAF-invariant and with associated permutation $\sigma\in \mathcal S_3$.

If $\sigma$ fixes $1$ or $3$ then $f$ is a restricted rotation with zero SAF-invariant. Therefore, by Properties \ref{PropSAF} $f$ is periodic. 

We suppose that $\sigma$ does not fix $1$ and $3$. 

If $\sigma$ is a 3-cycle then $f$ is a rotation  with zero SAF-invariant. As before, we conclude that $f$ is periodic. 

If $\sigma$ is a 2-cycle that does not fix $1$ and $3$ then $\sigma=(1,3)$.

Therefore, $f$ has 2 discontinuities $a_1$ and $a_2$. We set $a_0=0$, $a_3=1$ and we denote by $\lambda_i$ the length of $I_i=[a_{i-1}, a_{i})$ for $i=1,2,3$. 

One has  $a_1= \lambda_1$, $a_2=\lambda_1+\lambda_2$ and $f$ is easily computed:

\begin{itemize}
\item $f(x) = x + \lambda_2+\lambda_3= x + 1- \lambda_1$ \qquad  \qquad  for $x\in I_1$,
\item $f(x) = x - \lambda_1+\lambda_3=  x + 1-2\lambda_1-\lambda_2$  \ \quad for $x\in I_2$,
\item $f(x) = x - \lambda_1-\lambda_2$ \qquad \qquad  \qquad \qquad \qquad for $x\in I_3$.
\end{itemize}
\medskip

We compute the SAF-invariant of $f$: $$SAF(f)=\lambda_1 \otimes  ( 1-\lambda_1)  +    \lambda_2 \otimes (1-2\lambda_1-\lambda_2) +   (1-\lambda_1-\lambda_2) \otimes ( - \lambda_1-\lambda_2)=$$
 $$ \lambda_1 \otimes 1 + \lambda_2 \otimes (1-2\lambda_1) 
- 1\otimes \lambda_1  +\lambda_1\otimes \lambda_2  + 1 \otimes (-\lambda_2) + \lambda_2\otimes \lambda_1= $$
$$\lambda_1\wedge\lambda_2 + (\lambda_1+\lambda_2) \wedge 1 = - \lambda_1\wedge\lambda_1 - \lambda_2\wedge\lambda_1 + (\lambda_1+\lambda_2) \wedge 1 = $$ $$ - (\lambda_1+\lambda_2)\wedge\lambda_1 + (\lambda_1+\lambda_2) \wedge 1 = (\lambda_1+\lambda_2)\wedge (1-\lambda_1 ). $$
 
Therefore, $SAF(f)=0$ if and only if  $\frac{ \lambda_1+\lambda_2}{1-\lambda_1 } \in \mathbb Q$.
 
\medskip

It is straightforward to check that the first return map of $f$ on $[0, \lambda_1+\lambda_2)$ is 
given by $x\mapsto x+ (1-2\lambda_1-\lambda_2) \ mod  (\lambda_1+\lambda_2)$. Therefore it is minimal if and only if $\frac {(1-2\lambda_1-\lambda_2)}{\lambda_1+\lambda_2} \notin \mathbb Q$.

Note that if $\frac {(1-2\lambda_1-\lambda_2)}{\lambda_1+\lambda_2} \in \mathbb Q$ then the first return map is periodic so is $f$.

Since $\frac {(1-2\lambda_1-\lambda_2)}{\lambda_1+\lambda_2}= -1 + \frac{1-\lambda_1 } { \lambda_1+\lambda_2}$ we conclude that $SAF(f)=0$ if and only if $f$ is periodic. \end{proof}

\subsubsection{The SAF invariant for  product of two restricted rotations.}

Given $(l_1, l_2,\delta_1,\delta_2)\in \mathbb R^4$ such that $l_1+l_2=1, 0\leq \delta_1 < l_1$  and $0\leq \delta_2 < l_2$.
We consider  $f$  the product of two restricted rotations with associated permutation $(1,2) (3,4)$ and length vector $(l_1-\delta_1, \delta_1, 1-\delta_2-l_1, \delta_2)$. The translation vector is $(\delta_1, \delta_1-l_1,\delta_2, \delta_2-l_2)$. An easy computation leads to  
\begin{equation} SAF(f)=l_1\wedge \delta_1 + l_2\wedge \delta_2. 
\label{SAF} \end{equation} 

\subsubsection{Item 2: A 4-IET with zero SAF-invariant that is not reversible.} \ 

\begin{proof}

We claim that if $(\frac{\delta_1}{l_1},\frac{\delta_2}{l_2}) \notin \mathbb Q^2$ and $l_1\not= l_2$ then $f$ is not reversible. Indeed, we argue by contradiction. Any $h$ that reverses $f$ sends minimal $f$-components to minimal ones. As $l_1\not= l_2$, the map $h$ preserves $[0,l_1)$ and its restriction reverses a minimal rotation, a contradiction. 

\medskip

Assuming that $l_1 = p l_2 \in \mathbb Q^*$ and $\delta_2 = -p\delta_1 +r$ for some rational numbers $p$ and $r$, we get $SAF(f)=p l_2\wedge \delta_1 + l_2\wedge \delta_2= l_2 \wedge (p \delta_1+ \delta_2) =l_2 \wedge r =r (l_2 \wedge 1)=0$.\end{proof}

\subsubsection{Item 3: Products of two restricted rotations with zero SAF-invariant.} \ 

\begin{proof}  From Formula (\ref{SAF}), it holds that $SAF(f)=l_1\wedge \delta_1 + (1-l_1)\wedge \delta_2 =0$.

\smallskip

Moreover if $l_1 = \frac{p}{p+1}$ then $l_2 = 1- l_1=\frac{1}{p+1}$ and then  $l_1=p l_2$.

Therefore  $SAF(f)=p l_2\wedge \delta_1 + l_2\wedge \delta_2= l_2 (1 \wedge (p \delta_1+ \delta_2))  =0$ implies that  $p \delta_1+ \delta_2 \in \mathbb Q$ that is $\delta_2 = -p \delta_1+r$, with $r\in \mathbb Q$.

\smallskip
 
The homothecy $H$ of ratio $p+1$ conjugates $f$ to $F$ a product of two restricted rotations whose supports have lengths $l_1=p$ and $l_2=1$ respectively.

\smallskip

Let $RP$ be the product of the two restricted rotations, defined by:
\begin{itemize}
\item $RP(x)= x-q \ mod \ p$, on $[0, p)$, where $q \in \mathbb Q^+$ is such that $0<\delta_1-q<\frac{1}{p}$ and
\item $RP(x)= x-r+pq \ mod \ 1$,  on $[p, p+1)$.
\end{itemize}
As $p$, $q$ and $r$ are rational numbers, $RP$ is periodic.
 
\medskip
 
We have that $g= F\circ RP$ is the product of two restricted rotations and its length vector is $(p-\alpha,\alpha, p\alpha, 1-p\alpha)$ where $\alpha=\delta_1-q \in (0,\frac{1}{p})$. Hence, the translation vector of $g$ is $(\alpha,\alpha -p, 1-p\alpha, -p\alpha)$.
 
\smallskip
 
We suppose that $p\in \mathbb N^*$, we consider the group $\hat{G}_{p+1}$ the conjugate of $G_{p+1}$ by $H$ and we define $i$ as the element $\hat{G}_{p+1}$ with $\sigma_i = (1,2,...,p+1)$ and 
$\alpha_i = (-\alpha,-\alpha,...,-\alpha,p\alpha)$. An easy computation shows that $BP (i) = \{p+1-p\alpha, p, j, j+\alpha {\text{ with }} j=0,...,p-1 \}$ and $i$ is periodic of period $p+1$.

\medskip

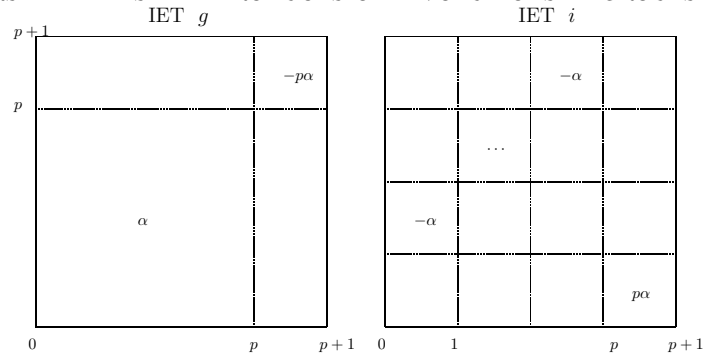
\begin{figure}[H] \scalebox{.55}{\input{Inv832.tex}}\end{figure}

\medskip

It follows that $[0, 1-\alpha)$, $[p-\alpha,p)$ , $[p,p+1)$, $[j-\alpha, j)$, $[j,j+1-\alpha)$ with $j=1,...,p-1$ are the continuity intervals of $i\circ g$ and its translation constants are integers. We conclude that $i\circ g$ is periodic. Hence $g$ is the product of two periodic IETs, therefore $g$ can be written as the product of $4$ involutions and $F=g \circ RP^{-1}=$ is the the product of $6$ involutions: $i_k$, $k=1,\cdots ,6$. 

Finally, the initial map $f=H \circ  F \circ  H^{-1}$ is the product of $H \circ i_k \circ H^{-1}$,  $k=1,\cdots,6$ that are involutions of $\mathcal G$.\end{proof}

\section{Related groups. Proof of Theorem \ref{related}.}
As pointed out by Leroux (\cite{Ler}), reversible elements occur in some torsion-free groups, namely $BS(1,-1)$, the crystallographic group $C_1$ and the group $C_2$ defined in Theorem \ref{related}. Note that, by Theorem \ref{Th4.1} (2), in faithful actions of these groups reversible elements are not minimal. It is natural to  ask whether these groups admit free or/and minimal actions by IET. 
\subsection{The Baumslag-Solitar group $BS(1,-1)$. Proof of item (1).} \ 

In the example of section 2, the set $[0,\frac{1}{4})\cup [\frac{3}{4},1)$ is invariant, so the action is not minimal. Here, we give an example where $ BS(1,-1)$ acts freely and minimally by IET.

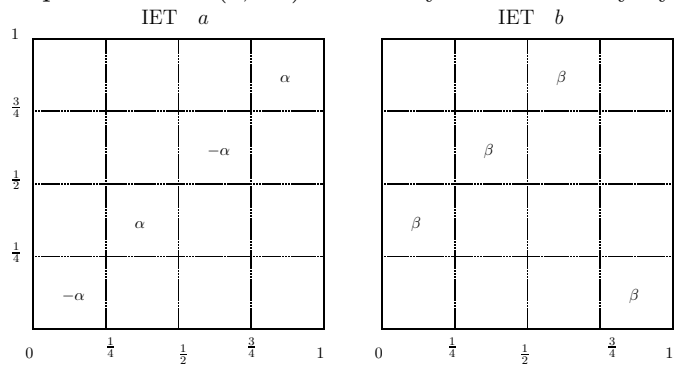
\begin{figure}[H]\scalebox{.55}{\input{BSfree_min.tex}}\end{figure}

\bigskip

This action is free and minimal provided that $\alpha$ and $\beta$ are rationally independent.

\subsection{The crystallographic group $C_1$.  Proof of item  (2).} \ 

The following action of  $C_1 =\langle a, b \  |  \ b a^2  b^{-1} = a^{-2},  a b^2 a^{-1} =b^{-2}\rangle$ is free and minimal, provided that $\alpha$ and $\beta$ are rationally independent.

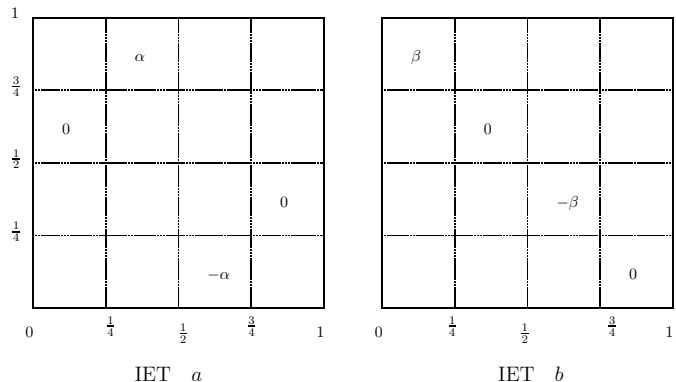
\begin{figure}[H] \scalebox{.55}{\input{BS1.tex}}
\end{figure}

\bigskip  \ 

\subsection{The group $C_2$. Proof of item (3).} \ 

Leroux (\cite{Ler}) has proved that the elements $a$, $b$, $c$ are non trivial, $C_2=\langle a, b, c \ | \  ab a^{-1} = b^{-1}, bcb^{-1} = c^{-1} \rangle$ is torsion-free, every element of $C_2$ has a unique expression of the form $\omega b^n$, with $\omega \in F_2$ and the subgroup $F_2$ of $C_2$ generated by $a$ and $c$ is free. 

We prove that $C_2$ does not admit faithful  minimal action by IET. More precisely, we prove 
\begin{prop} Any  faithful action of $C_2$ by IET is conjugated in $\mathcal G$ to a reducible representation $\Gamma_2 < \mathcal G$ such that there exist a partition of $I$ into two half open $\Gamma_2$-invariant intervals $I_1$ and $I_2$ and non zero integers $l$, $n$, $p$  verifying that 
\begin{itemize}
\item $[a^l, c^n]=Id$ in $I_1$ and
\item $b^p=Id$ on $I_2$.
\end{itemize}
\end{prop}

\begin{proof} According to the Arnoux-Keane-Mayer decomposition Theorem, $I$ can be written as a finite union of $c$-minimal or $c$-periodic components. Eventually conjugating $c$ in  $\mathcal G$, we can suppose that $c$-components are intervals.

Let $M$ be a minimal $c$-component, Lemma \ref{comp} implies that the set $b(M)$ is a minimal component of $c$. Thus, by finiteness, there exists $r$ such that $b^r (M)=M$ and w.l.o.g we can assume that $r$ is even. 

\smallskip

The action of $\langle b^r,c\rangle$ on $I$ is not free. Indeed, by contradiction and according to Theorem \ref{Th3}, the map $c$ would be conjugate (in PL $\circ$ IET) to an element of some $G_n$ and therefore an iterate of $c$ would be a product of restricted rotations with pairwise disjoint supports, this is a contradiction with a result of Novak claiming that such an IET do not belong to a non abelian free group (\cite{No2} or \cite{DFG}, Theorem 3.6). 

Hence, there exist $x\in M$ and $(p,q) \in \mathbb Z ^2\setminus \{(0,0)\}$ such that $g(x)=b^{rp}c^q(x)=x$. As $b^{rp}$ and $c$ commute, $g$ and $c$ commute and then $Fix(g)$ is $c$-invariant 
Thus the closure of the $c$-orbit of $x$ is included in $Fix(g)$, that is $M\subset Fix(g)$. 

Noting that $b^2$ and $c$ commutes, we conclude that $g=b^{rp}c^q=Id$ on $\cup b^{2k} (M)$.

Using that $b$ reverses $c^q$, we have that $b^{rp}c^{-q}=Id$ on $\cup b^{2k+1} (M)$.

\smallskip

Let $B=b^{rp}$, one has $B=c^{-q}$ on $\cup b^{2k } (M)$ and $B=c^q$ on $\cup b^{2k+1}(M)$. Since $a$ reverses $B$, by Lemma \ref{comp}, there exists $l=l(M)$ even such that $A=a^l$ preserves the $B$-components and therefore the minimal $c$-components.  We compute $[A,c^{-q^2}]$: 

\begin{itemize}
\item  on $\cup b^{2k } (M)$, we get $[A,c^{-q^2}]=[A,(c^{-q} )^q]=[A,B^q]= Id$ and
\item  on $\cup b^{2k +1 } (M)$, we get $[A,c^{-q^2}]=[A,(c^{q} )^{-q}]=[A,B^{-q}]= Id$.
\end{itemize}

In conclusion, given $M$ a minimal $c$-component, either $M$ is of
\begin{enumerate}[{\it type} (1):]
\item there exists $n=n(M)\not=0$ such that $[A,c^{n}]=Id$ on  $\cup b^{k} (M)$ (when $q\not= 0$)   or
\item there exists $p=p(M)\not=0$ such that $b^{p}=Id$ on  $\cup b^{k} (M)$ (when $q= 0$).
\end{enumerate}
 
Moreover, there is $t\not=0$ such that $c^t=Id$ on the union of the periodic components of $c$.

\smallskip

Changing $l$ by the product of all $l(M)$ and $n$ by the product of all $n(M)$ and $t$,  we have $[a^l,c^{n}]=Id$ on $I_1$,  the union of the periodic and  minimal $c$-components of type (1).

Changing $p$ by the product of all $p(M)$,  we have $b^p=Id$ on $I_2$, the union of the minimal $c$-components of type (2).
Conjugating by an IET, we arrive at the case where $I_1$ and $I_2$ are intervals.
\end{proof}
\bibliographystyle{alpha}
\bibliography{RefIET}
\end{document}

%% file: BS.tex
\begin{picture}(450,200)
\put(0,0){\line(1,0){200}} \put(0,200){\line(1,0){200}}
\put(0,0){\line(0,1){200}} \put(200,0){\line(0,1){200}}

\put(-5,-20){$0$}
\put( 50,-15){$\frac{1}{4}$}
\put(100,-20){${\frac{1}{2}}$}
\put(147,-15){$\frac{3}{4}$}
\put(195,-20){$1$}

\put(-15,50){$\frac{1}{4}$}
\put(-15,100){$\frac{1}{2}$}
\put(-15,150){$\frac{3}{4}$}
\put(-15,200){$1$}

\put(50,0){\dashbox(0,200){}}
\put(100,0){\dashbox(0,200){}}
\put(150,0){\dashbox(0,200){}}

\put(0,50){\dashbox(200,0){}}
\put(0,100){\dashbox(200,0){}}
\put(0,150){\dashbox(200,0){}}

\put (20,20){$-\alpha$}
\put (70,70){$\alpha$}
\put (120,120){$-\alpha$}
\put (170,170){$\alpha$}

\put ( 70,-50){IET \ \ $a$}

\put(240,0){\line(1,0){200}} \put(240,200){\line(1,0){200}}
\put(240,0){\line(0,1){200}} \put(440,0){\line(0,1){200}}

\put(235,-20){$0$}
\put(285,-15){${\frac{1}{4}}$}
\put(335,-20){${\frac{1}{2}}$}
\put(395,-15){${\frac{3}{4}}$}
\put(435,-20){$1$}


\put(290,0){\dashbox(0,200){}}
 \put(340,0){\dashbox(0,200){}}
\put(390,0){\dashbox(0,200){}}

\put(240,50){\dashbox(200,0){}} 
\put(240,100){\dashbox(200,0){}}
\put(240,150){\dashbox(200,0){}}

\put (260,170){$0$}
\put (310,120){$\beta$}
\put (360,70){$0$}
\put (410,20){$-\beta$}

\put ( 320,-50){IET \ \ $b$}
\end{picture}

%% file: Inv832.tex
\begin{picture}(450,200)
\put(0,0){\line(1,0){200}} \put(0,200){\line(1,0){200}}
\put(0,0){\line(0,1){200}} \put(200,0){\line(0,1){200}}

\put(-5,-15){$0$}
\put(147,-15){$p$}
\put(195,-15){$p+1$}

\put(-15,150){$p$}
\put(-15,200){$p+1$}

\put(150,0){\dashbox(0,200){}}

\put(0,150){\dashbox(200,0){}}

\put (70,70){$\alpha$}
\put (170,170){$-p\alpha$}

\put ( 70,210){ \ \Large{IET \ $g$}} 

\put(240,0){\line(1,0){200}} \put(240,200){\line(1,0){200}}
\put(240,0){\line(0,1){200}} \put(440,0){\line(0,1){200}}

\put(235,-15){$0$}
\put(285,-15){$1$}

\put(395,-15){$p$}
\put(435,-15){$p+1$}


\put(290,0){\dashbox(0,200){}}
 \put(340,0){\dashbox(0,200){}}
\put(390,0){\dashbox(0,200){}}

\put(240,50){\dashbox(200,0){}} 
\put(240,100){\dashbox(200,0){}} 
\put(240,150){\dashbox(200,0){}}

\put (260,70){$-\alpha$}
\put (360,170){$-\alpha$}
\put (310,120){$\cdots$}
\put (410,20){$p\alpha$}

\put ( 320,210){\ \ \Large{ IET \ $i$}}
\end{picture}

%% file: BSfree_min.tex
\begin{picture}(450,200)
\put(0,0){\line(1,0){200}} \put(0,200){\line(1,0){200}}
\put(0,0){\line(0,1){200}} \put(200,0){\line(0,1){200}}

\put(-5,-20){$0$}
\put( 50,-15){$\frac{1}{4}$}
\put(100,-20){${\frac{1}{2}}$}
\put(147,-15){$\frac{3}{4}$}
\put(195,-20){$1$}

\put(-15,50){$\frac{1}{4}$}
\put(-15,100){$\frac{1}{2}$}
\put(-15,150){$\frac{3}{4}$}
\put(-15,200){$1$}

\put(50,0){\dashbox(0,200){}}
\put(100,0){\dashbox(0,200){}}
\put(150,0){\dashbox(0,200){}}

\put(0,50){\dashbox(200,0){}}
\put(0,100){\dashbox(200,0){}}
\put(0,150){\dashbox(200,0){}}

\put (20,20){$-\alpha$}
\put (70,70){$\alpha$}
\put (120,120){$-\alpha$}
\put (170,170){$\alpha$}

\put ( 70,210){\Large{ IET \ \ $a$}}

\put(240,0){\line(1,0){200}} \put(240,200){\line(1,0){200}}
\put(240,0){\line(0,1){200}} \put(440,0){\line(0,1){200}}

\put(235,-20){$0$}
\put(285,-15){${\frac{1}{4}}$}
\put(335,-20){${\frac{1}{2}}$}
\put(395,-15){${\frac{3}{4}}$}
\put(435,-20){$1$}


\put(290,0){\dashbox(0,200){}}
 \put(340,0){\dashbox(0,200){}}
\put(390,0){\dashbox(0,200){}}

\put(240,50){\dashbox(200,0){}} 
\put(240,100){\dashbox(200,0){}}
\put(240,150){\dashbox(200,0){}}

\put (260,70){$\beta$}
\put (310,120){$\beta$}
\put (360,170){$\beta$}
\put (410,20){$\beta$}

\put ( 320,210){\Large{IET \ \ $b$}}
\end{picture}

%% file: BS1.tex
\begin{picture}(450,200)
\put(0,0){\line(1,0){200}} \put(0,200){\line(1,0){200}}
\put(0,0){\line(0,1){200}} \put(200,0){\line(0,1){200}}

\put(-5,-20){$0$}
\put( 50,-15){$\frac{1}{4}$}
\put(100,-20){${\frac{1}{2}}$}
\put(147,-15){$\frac{3}{4}$}
\put(195,-20){$1$}

\put(-15,50){$\frac{1}{4}$}
\put(-15,100){$\frac{1}{2}$}
\put(-15,150){$\frac{3}{4}$}
\put(-15,200){$1$}

\put(50,0){\dashbox(0,200){}}
\put(100,0){\dashbox(0,200){}}
\put(150,0){\dashbox(0,200){}}

\put(0,50){\dashbox(200,0){}}
\put(0,100){\dashbox(200,0){}}
\put(0,150){\dashbox(200,0){}}

\put (20,120){$0$}

\put (70,170){$\alpha$}

\put (120,20){$- \alpha$}

\put (170,70){$0$}

\put ( 70,-50){\Large{IET \ \ $a$}}

\put(240,0){\line(1,0){200}} \put(240,200){\line(1,0){200}}
\put(240,0){\line(0,1){200}} \put(440,0){\line(0,1){200}}

\put(235,-20){$0$}
\put(285,-15){${\frac{1}{4}}$}
\put(335,-20){${\frac{1}{2}}$}
\put(395,-15){${\frac{3}{4}}$}
\put(435,-20){$1$}


\put(290,0){\dashbox(0,200){}}
\put(340,0){\dashbox(0,200){}}
\put(390,0){\dashbox(0,200){}}

\put(240,50){\dashbox(200,0){}} 
\put(240,100){\dashbox(200,0){}}
\put(240,150){\dashbox(200,0){}}

\put (260,170){$\beta$}
\put (310,120){$0$}
\put (360,70){$-\beta$}
\put (410,20){$0$}

\put ( 320,-50){\Large{IET \ \ $b$}}
\end{picture}
\bigskip